\documentclass[a4paper]{article}

\usepackage{amssymb,amsfonts,amsmath,amsthm,cite,mathrsfs}
\usepackage{graphicx,epsfig,subfigure}
\usepackage{enumerate}
\usepackage{dsfont}
\usepackage{color,xcolor}

\newtheorem{thm}{Theorem}[section]

\newtheorem{cor}{Corollary}[section]

\newtheorem{lem}{Lemma}[section]
\newtheorem{claim}{Claim}[section]
\newtheorem{prop}{Proposition}[section]

\newtheorem{ex}{Example}[section]

\newtheorem*{claim*}{Claim}

\makeatletter \@addtoreset{equation}{section}
\def\qed{\hfill \rule{4pt}{7pt}}
\def\hf{\mathcal{F}}
\def\hk{\mathcal{K}}

\def\hm{\mathcal{M}}
\def\hg{\mathcal{G}}
\def\hht{\mathcal{T}}
\def\hh{\mathcal{H}}
\def\hhn{\mathcal{N}}
\def\hp{\mathcal{P}}

\def\hc{\mathcal{C}}
\def\hext{\mathcal{H}_{ext}}
\def\ha{\mathcal{A}}
\def\hb{\mathcal{B}}
\def\ohb{\overline{\mathcal{B}}}
\def\hl{\mathcal{L}}

\def\hn{\mathbb{N}}
\def\ex{\mathcal{\mathbb{E}}}
\def\hr{\mathcal{R}}

\begin{document}

\title{Minimum degree thresholds for Hamilton $(\ell,k-\ell)$-cycles in $k$-uniform hypergraphs}
\author{Jian Wang$^1$, Jie You$^2$\\[5pt]
$^1$Department of Mathematics\\
Taiyuan University of Technology\\
Taiyuan 030024, P. R. China\\[6pt]
$^2$Center for Applied Mathematics\\
Tianjin University\\
Tianjin 300072, P. R. China\\[6pt]
E-mail:  $^1$wangjian01@tyut.edu.cn, $^2$yj\underline{ }math@tju.edu.cn
}

\date{}
\maketitle
\begin{abstract}
Let $n>k>\ell$ be positive integers. We say a $k$-uniform hypergraph $\hh$ contains a Hamilton $(\ell,k-\ell)$-cycle if there is a partition $(L_0,R_0,L_1,R_1,\ldots,L_{t-1},R_{t-1})$ of  $V(\hh)$ with $|L_i|=\ell$, $|R_i|=k-\ell$ such that $L_i\cup R_i$ and $R_i\cup L_{i+1}$ (subscripts module $t$) are all edges of $\hh$ for $i=0,1,\ldots,t-1$. In the present paper, we determine the  tight minimum $\ell$-degree condition that guarantees the existence of a Hamilton $(\ell,k-\ell)$-cycle   in every $k$-uniform $n$-vertex hypergraph  for $k\geq 7$, $k/2\leq \ell\leq k-1$ and sufficiently large $n\in k\hn$.
\end{abstract}

\section{Introduction}
\subsection{Background}
A central question in graph theory is to establish conditions that ensure a (hyper)graph $\hh$ contains some spanning (hyper)graph $\hf$.  Of course, it is desirable to fully characterize those (hyper)graphs $\hh$ that contain a spanning copy of a given (hyper)graph $\hf$.  A theorem of Tutte \cite{tutte1947factorization}  gives a characterization  of all those graphs which contain a perfect matching.  However, for some (hyper)graphs $\hf$ it is unlikely that such a characterization exists. Indeed, for many (hyper)graphs $\hf$ the decision problem of whether a (hyper)graph $\hh$ contains $\hf$ is NP-complete.  Garey and Johnson \cite{garey1979computers} showed that the decision problem whether a $k$-uniform hypergraph contains a perfect matching or Hamilton path is NP-complete for $k\geq3$. It is natural therefore to seek simple sufficient conditions that ensure a Hamilton cycle in a $k$-uniform hypergraph. The study of Hamilton cycles is an important topic in graph theory with a long history. In 1952, Dirac \cite{Dirac1952} proved that if $\delta(G)\geq n/2$ and $n\geq 3$ then $G$ contains a Hamilton cycle, which is one of the most classical results in graph theory.  In recent years, researchers have worked on extending Dirac's theorem to hypergraphs and we refer to
\cite{OliveiraBastos2017,bastos2018loose,buss2013minimum,glebov2012extremal,Han2015,Han2016,Reiher2019,Roedl2010}
for some recent results and to \cite{Kuehn2014,Roedl2010,Zhao2016}
for surveys on this topic.

Given a set $V$ of size $n$ and an integer $k\geq 2$, we use $\binom{V}{k}$ to denote the family of all $k$-element subsets ($k$-subsets, for short) of $V$. A subfamily $\hh\subset\binom{V}{k}$ is called  a {\it $k$-uniform hypergraph} (or {\it $k$-graph} in short). For $\hh\subset\binom{V}{k}$, we often use $V(\hh)$ to denote its vertex set $V$ and use $\hh$ to denote its edge set. Define  the complement of $\hh$ as $\overline{\hh}:=\binom{V}{k}\setminus \hh$. Given $A\subseteq V$, let $\hh[A]$ denote the sub $k$-graph of $\hh$ induced by $A$, namely, $\hh[A]:=\hh\cap \binom{A}{k}$. Define $\hh- A:=\hh[V(\hh)\setminus A]$. For $S\in \binom{V}{\ell}$ with $0\leq \ell\leq k-1$, define the {\it link graph} of $S$ $\hhn_\hh(S):=\{T\colon S\cup T\in \hh\}$ and let $\deg_\hh(S)$  be the cardinality of $\hhn_\hh(S)$.  The  \emph{minimum $\ell$-degree} $\delta_\ell(\hh)$ of $\hh$ is the minimum of $\deg_\hh(S)$ over all $\ell$-element subsets $S$ of $V(\hh)$. Clearly $\delta_0(\hh)$ is the number of edges in $\hh$. We refer to $\delta_1(\hh)$ as the \emph{minimum vertex degree} of $\hh$ and $\delta_{k-1}(\hh)$ the \emph{minimum codegree} of $\hh$. We often omit the subscript $\hh$ when the context is clear.

 Let $n,k,\ell$ be positive integers with $\ell<k$ and $(k-\ell)|n$. A $k$-graph is called an {\it $\ell$-cycle} if there is a cyclic ordering of the vertices such that every edge consists of $k$ consecutive vertices, every vertex is contained in an edge and two consecutive edges (where the ordering of the edges is inherited from the ordering of the vertices) intersect in exactly $\ell$-vertices. A $(k-1)$-cycle is also called a {\it tight cycle}. We say a $k$-graph $\hh$ contains a Hamilton $\ell$-cycle if there is a subhypergraph of $\hh$ that forms an $\ell$-cycle and covers all vertices of $\hh$.

Confirming a conjecture of Katona and Kierstead \cite{katona1999hamiltonian}, R\"{o}dl, Ruci\'{n}ski and Szemer\'{e}di \cite[etc]{VOJTECH2006,Roedl2008} showed that for any fixed $k$, every $k$-graph $\hh$ on $n$ vertices with $\delta_{k-1}(\hh)\geq n/2+o(n)$ contains a Hamilton tight cycle. This is best possible up to the $o(n)$ term by a construction given by Katona and Kierstead \cite{katona1999hamiltonian}. R\"{o}dl, Ruci\'{n}ski and Szemer\'{e}di \cite{Roedl2011} eventually determined the minimum codegree threshold for Hamilton tight cycles in $3$-graphs for sufficiently large $n$, which is $\lfloor n/2\rfloor$.

After a series of efforts \cite{Keevash2011,Han2010,Kuehn2010}, the minimum codegree conditions for the existence of Hamilton $\ell$-cycle were determined asymptotically. R\"{o}dl and Ruci\'{n}ski \cite[Problem 2.9]{Roedl2010} raised the question concerning the \emph{exact} minimum codegree condition for Hamilton $\ell$-cycle. The case $k = 3$ and $\ell= 1$ was solved by Czygrinow and Molla \cite{Czygrinow2014}. The threshold for all $k\geq 3$ and $\ell<k/2$ was determined by Han and Zhao \cite{Han2015a}. The case $k = 4$ and $\ell= 2$ was determined by Garbe and Mycroft \cite{Garbe2018}. Recently, the case  $k\geq 6$, $k$ is even and $\ell= k/2$ was determined by H\`{a}n, Han and Zhao \cite{Han2022}.

A $k$-graph $\hh$ is called an {\it $(\ell,k-\ell)$-cycle} if there is a partition $(L_0,R_0,L_1,R_1,\ldots,L_{t-1},R_{t-1})$ of  $V(\hh)$ with $|L_i|=\ell$, $|R_i|=k-\ell$ such that $L_i\cup R_i$ and $R_i\cup L_{i+1}$ (subscripts module $t$) are all edges of $\hh$ for $i=0,1,\ldots,t-1$. Similarly, we say $\hh$ contains a Hamilton $(\ell,k-\ell)$-cycle if there is a subhypergraph of $\hh$ that forms an $(\ell,k-\ell)$-cycle and covering all vertices of $\hh$. In \cite{Wang2023}, the minimum $\ell$-degree condition for the existence of Hamilton $(\ell,k-\ell)$-cycle were determined asymptotically for $\ell\geq k/2$.

A $k$-graph $\hp$ is called an {\it $(\ell,k-\ell)$-path} if there is a partition $(P_0,P_1,P_2,\ldots,P_t)$ of  $V(\hp)$  such that $P_i\cup P_{i+1}$ is an edge of $\hp$ for $i=0,1,\ldots,t-1$ and either $|P_i|=\ell$, $|P_{i+1}|=k-\ell$ or  $|P_i|=k-\ell$, $|P_{i+1}|=\ell$. We call $P_0,P_t$ the {\it ends} of this path.

\subsection{Main results and relative construction}

Given a set $V$ of $n$ vertices and a partition $V=A\cup B$, let $E_{odd}(A,B)$ $(E_{even}(A,B))$ denote the family of all $k$-element subsets of $V$ that intersect $A$ in an odd (even) number of vertices. Define $\mathcal{B}_{n,k}(A,B)$ to be the $k$-graph with the vertex set $V=A\cup B$ and the edge set $E_{odd}(A,B)$. The complement $\overline{\hb}_{n,k}(A,B)$ has edge set $E_{even}(A,B)$. Let $\hext(n,k)$ be the family of $k$-graphs containing all hypergraphs $\hb_{n,k}(A,B)$ when $n/k-|A|$ is odd and all $\ohb(A,B)$ when $|A|$ is odd. Let \[\delta(n,k,\ell)=\max_{\hh\in\hext(n,k)}\delta_\ell(\hh).\]

\begin{thm}[\hspace{1sp}\cite{Treglown2012, Treglown2013}]\label{ZT}
Given integers $k,\ell$ such that $k\geq 3$ and $k/2\leq \ell\leq k-1$, there exists  $n_0\in \hn$ such that the following holds. Suppose that $\hh$ is a $k$-graph on $n$ vertices with  $n\in k\hn$ and $n\geq n_0$ satisfying $\delta_\ell(\hh)>\delta(n,k,\ell)$, then $\hh$ contains a perfect matching.
\end{thm}

In the present paper, we mainly prove the following result.

\begin{thm}[Main Result]\label{thm-main}
  Given integers $k,\ell$ such that $k\geq 7$ and $k/2\leq \ell\leq k-1$, there exists  $n_0\in \hn$ such that the following holds. Suppose that $\hh$ is a $k$-graph on $n$ vertices with  $n\in k\hn$ and $n\geq n_0$ satisfying $\delta_\ell(\hh)>\delta(n,k,\ell)$, then $\hh$ contains a Hamilton $(\ell,k-\ell)$-cycle.
\end{thm}

Note that a Hamilton $(\ell,k-\ell)$-cycle consists of two perfect matchings.  Our result can be viewed as a strengthening of Theorem \ref{ZT} for $k\geq 7$. It should be mentioned that the case for  $k\geq 6$ is even and $\ell= k/2$ was already proved by H\`{a}n, Han and Zhao \cite{Han2022}, they also determined the case when $n\in \frac{k}{2}\hn$ and $\ell= k/2$.

It is proved in \cite{Treglown2012} that no hypergraph in $\hext(n,k)$ contains a perfect matching. We infer that  no hypergraph in $\hext(n,k)$ contains a $(\ell,k-\ell)$-cycle for $1\leq \ell\leq k-1$. It implies that the minimum $\ell$-degree condition in Theorem \ref{thm-main} is best possible.

In \cite{Treglown2012}, it is showed that
\begin{equation}
	\delta(n,k,k-1)=
	\begin{cases}
		n/2-k+2 & \mbox{if $k/2$ is even and $n/k$ is odd}\\
		n/2-k+3/2 & \mbox{if $k$ is odd and $(n-1)/2$ is odd }\\
		n/2-k+1/2 & \mbox{if $k$ is odd and $(n-1)/2$ is even}\\
        n/2-k+1 & \mbox{otherwise.}
	\end{cases}
\end{equation}
It seems hard to compute the precise values of $\delta(n,k,\ell)$ for $\ell\leq k-2$, which are  only known to be $(1/2+o(1))\binom{n-\ell}{k-\ell}$, see \cite{Treglown2012} for details.

\subsection{Proof of Theorem \ref{thm-main}}

As a common approach to obtain exact results, Theorem \ref{thm-main} is proven by distinguishing an \emph{extremal} case from a \emph{non-extremal} case and solve them separately.

Let $\epsilon>0$ and suppose that $\hh$ and $\hh'$ are $k$-graphs on $n$ vertices. We say that $\hh$ is $\epsilon$-close to $\hh'$, and write $\hh=\hh'\pm\epsilon n^k$, if $\hh$ can be made a copy of $\hh'$ by adding and deleting at most $\epsilon n^k$ edges. If $\hh$ is a $k$-graph with $\delta_\ell(\hh)\geq (1/2-o(1))\binom{n-\ell}{k-\ell}$ and $o(1)$-close to some $k$-graph in $\hext(n,k)$, then $\hh$ must be $o(1)$-close to some $\hb_{n,k}(A,B)$ or $\ohb_{n,k}(A,B)$ with $|A|=\lceil n/2\rceil, |B|=\lfloor n/2\rfloor$ as well (cf. \cite{Han2022}). In the following we simply write $\hb_{n,k}$ and $\ohb_{n,k}$ to indicate that there is an implicit partition $A\cup B$ of almost equal size.

\begin{thm}[Extremal Case]\label{thm-extr}
Given integer $k,\ell$ with  $k/2\leq \ell\leq k-1$ and $k\geq 7$, there exist real   $\epsilon>0$  and integer $n_0=n_0(k,\epsilon)$ such that the following holds. Let $\hh$ be a $k$-graph on $n$ vertices with $n\geq n_0$ and $n\in k\hn$. If $\hh$ is $\epsilon$-close to any $\hb_{n,k}$ or $\ohb_{n,k}$ and $\delta_\ell(\hh)> \delta(n,k,\ell)$, then $\hh$ contains a Hamilton $(\ell,k-\ell)$-cycle.
\end{thm}

\begin{thm}[Non-extremal Case]\label{thm-nonextr}
	Given real  $\epsilon>0$ and integers $k, \ell$ with  $k\geq 3$ and $k/2\leq \ell\leq k-1$,  there exist $n_0\in\hn$ and real  $\gamma>0$ such that  the following holds. Let $\hh$ be a $k$-graph on $n$ vertices with $n\geq n_0$ and $n\in k\hn$. If $\hh$ is not $\epsilon$-close to any $\hb_{n,k}$ or $\ohb_{n,k}$ and $\delta_\ell(\hh)\geq (1/2-\gamma)\binom{n-\ell}{k-\ell}$, then $\hh$ contains a Hamilton $(\ell,k-\ell)$-cycle.
\end{thm}

\subsection{Preliminaries and notations}
The follow result concerning $\ell$-degrees and $\ell'$-degrees are useful in our proof, which is also used in \cite{Treglown2012, Treglown2013}.

\begin{prop}\label{prop-2.1}
	Let $0\leq \ell'\leq \ell<k$ and let $\hh$ be a $k$-graph. If $\delta_\ell(\hh)\geq x\binom{n-\ell}{k-\ell}$ for some $0\leq x\leq 1$, then $\delta_{\ell'}(\hh)\geq x\binom{n-\ell'}{k-\ell'}$.
\end{prop}

We need a concentration inequality due to Frankl and Kupavskii.

\begin{thm}[Frankl-Kupavskii Concentration Inequality, \cite{frankl2018erd}]\label{thm-concentrate}
	Suppose that $m,k,t$ are integers and $m\geq tk$. Let $\hg\subset \binom{[m]}{k}$ be a family, and $\theta=|\hg|/\binom{[m]}{k}$. Let $\eta$ be the random variable equal to the size of the intersection of $\hg$ with a $t$-matching $\hb$ of $k$-sets, chosen uniformly at random. Then $\ex[\eta]=\theta t$ and, for any positive $\gamma$, we have
	\begin{align}\label{FK-ineq}
 \Pr[|\eta-\theta t|\geq2\gamma\sqrt{t}]\leq 2e^{-\gamma^2/2}.
 \end{align}
\end{thm}

The following version of the Chernoff bound
for binomial distributions is also needed (see e.g. \cite[Corollary 2.3]{Janson2000}). Recall that the binomial random variable with parameters $(n,p)$ is the sum of $n$ independent Bernoulli variables, each taking value $1$ with probability $p$ or $0$ with probability $1-p$.
\begin{prop}\label{prop:1.2}
	Suppose $X$ has binomial distribution and $0<a<3/2$. Then $\Pr(|X-\ex X|\geq a\ex X)\leq 2e^{-\frac{a^2}{3}\ex X}$.
\end{prop}

We often write $0<a_1\ll a_2\ll a_3$ to mean that we can choose the constants $a_1, a_2, a_3$ from right to left. More precisely, there are increasing functions $f$ and $g$ such that, given $a_3$, whenever we choose some $a_2\leq f(a_3)$ and $a_1\leq g(a_2)$, all calculations needed in our proof are valid. Hierarchies with more constants are defined in the obvious way. Throughout the paper we omit floors and ceilings whenever this does not affect the argument.

The proofs of Theorems \ref{thm-extr} and \ref{thm-nonextr}  are shown in Section \ref{sec:non-extrcase} and Section \ref{sec:extr case} separately.

\section{Non-extremal Case}\label{sec:non-extrcase}

In this section, we deal with the non-extremal case by following the absorbing method initiated by R\"{o}dl, Ruci\'{n}ski and Szemer\'{e}di \cite{VOJTECH2006}. We shall use the Frankl-Kupavskii Concentration Inequality to establish the Reservoir Lemma and the Absorbing Lemma.

 For disjoint sets $L$, $R$, $C$ with $|L|=\ell$, $|R|=k-\ell$, $|C|=2k$, if there exists an $(\ell,k-\ell)$-path on the vertex set  $L\cup R\cup C$ with ends  $L,R$, then we say $C$ connects $L,R$, or $C$ is a {\it connector} for $L,R$.

\begin{lem}[Reservoir Lemma]\label{lem-reservoir}
	For  integers $k,\ell$ with  $k\geq 3$, $k/2\leq \ell\leq k-1$, suppose that $1/n\ll\gamma\ll\epsilon\ll 1/k$. Let $\hh$ be a $k$-graph on $n$ vertices with $\delta_\ell(\hh)\geq(1/2-\gamma)\binom{n-\ell}{k-\ell}$. If $\hh$ is not $\epsilon$-close to $\hb_{n,k}$ or $\ohb_{n,k}$, then there exists  a family  $\hc$  with at most $\gamma^{19} n$ disjoint $2k$-sets  such that every $L,R \in V(\hh)$ with $|L|=\ell, |R|=k-\ell$ are connected by at least $\gamma^{23}n$ members in $\hc$.
\end{lem}

\begin{lem}[Absorbing Lemma]\label{lem-absorbing}
	For  integers $k,\ell$ with  $k\geq 3$, $k/2\leq \ell\leq k-1$, suppose that $1/n\ll\gamma\ll\epsilon\ll1/k$. Let $\hh$ be a $k$-graph on $n$ vertices with $\delta_\ell(\hh)\geq(1/2-\gamma)\binom{n-\ell}{k-\ell}$. If $\hh$ is not $\epsilon$-close to $\hb_{n,k}$ or $\ohb_{n,k}$, then there exists an $(\ell,k-\ell)$-path $\mathcal{P}$ in $\hh$ with $|V(\mathcal{P})|\leq \gamma n $ such that for all subsets $U\subset V\setminus V(\mathcal{P})$ of size at most $k\gamma^{18}n$  and $|U|\in k\hn$, there exists an $(\ell,k-\ell)$-path $\mathcal{Q}\subset \hh$ with $V(\mathcal{Q})=V(\mathcal{P})\cup U$ and, moreover, $\mathcal{P}$ and $\mathcal{Q}$ have the same ends.
\end{lem}
\begin{lem}[Path-cover Lemma]\label{lem-pathcover}
	For  integers $k,\ell$ with  $k\geq 3$, $k/2\leq \ell\leq k-1$, suppose that $1/n\ll1/p\ll\alpha\ll\gamma\ll1/k$. Let $\hh$ be a $k$-graph on $n$ vertices with $\delta_\ell(\hh)\geq(1/2-\gamma)\binom{n-\ell}{k-\ell}$. Then there are pairwise disjoint $(\ell,k-\ell)$-paths $\hp_1,\hp_2,\ldots,\hp_m$ in $\hh$   such that $m\leq p$, $|\cup_{1\leq i\leq m}V(\hp_i)|\geq (1-\alpha)n$ and each $\hp_i$ has ends $S_i,T_i$ with $|S_i|=\ell$, $|T_i|=k-\ell$ for $i=1,2,\ldots,m$.
\end{lem}

Now we ready to prove Theorem \ref{thm-nonextr}.

\begin{proof}[Proof of Theorem \ref{thm-nonextr}]
	Given  integers $k\geq 3$ and $k/2\leq \ell\leq k-1$,
	suppose $1/n\ll 1/p,\alpha\ll \gamma\ll\epsilon\ll 1/k$. Let $V$ be a set of size $n$ and let $\hh\subset \binom{V}{k}$  with $\delta_\ell(\hh)\geq (1/2-\gamma)\binom{n-\ell}{k-\ell}$. Assume that $\hh$ is not $\epsilon$-close to $\hb_{n,k}$ or $\overline{\hb}_{n,k}$.
	
	Since $\hh$ is not $\epsilon$-close to $\hb_{n,k}$ or $\overline{\hb}_{n,k}$, we can find an absorbing path $\mathcal{P}_0$ by Lemma \ref{lem-absorbing} with ends $S_0,T_0$ and $|V(\mathcal{P}_0)|\leq \gamma n$. Let $V_1=(V\setminus V(\mathcal{P}_0))\cup(S_0\cup T_0)$, we claim that $\hh[V_1]$ is not $(\epsilon/2)$-close to $\hb_{|V_1|,k}$ or $\overline{\hb}_{|V_1|,k}$. Suppose instead that there is a partition of $V_1=A\cup B$ with $|A|\leq |B|\leq |A|+1$ such that $\hh[V_1]$ is $(\epsilon/2)$-close to $\hb_{|V_1|,k}$ or $\overline{\hb}_{|V_1|,k}$. We add the vertices of $V\setminus V_1$ arbitrarily and evenly to $A$ and $B$, and get a partition of $V(\hh)=A'\cup B'$ with $|A'|=\lfloor n/2\rfloor$, $A\subseteq A'$, and $B\subseteq B'$. Since $|V\setminus V_1|\leq \gamma n$, we conclude that $\hh$ becomes a copy of $\hb_{n,k}(A',B')$ or $\overline{\hb}_{n,k}(A',B')$ after adding or deleting at most $\frac{\epsilon}{2}|V_1|^k+\gamma n\binom{n}{k-1}<\epsilon n^k$ edges because $\gamma\ll \epsilon$. This means that $\hh$ is $\epsilon$-close to $\hb_{n,k}$ or $\overline{\hb}_{n,k}$, a contradiction.
	
	Furthermore, as $|V\setminus V_1|\leq \gamma n$, we have
	\[
	\delta_\ell(\hh[V_1])\geq \left(\frac{1}{2}-\gamma\right)\binom{n-\ell}{k-\ell} -\gamma n \binom{n-\ell-1}{k-\ell-1}>  \left(\frac{1}{2}-k\gamma\right)\binom{|V_1|-\ell}{k-\ell}.
	\]
	 Apply Lemma \ref{lem-reservoir} on $\hh[V_1]$ and get a family $\hc$ of order $(k\gamma)^{19}|V_1|\leq (k\gamma)^{19}n$. Let $V_2:=V\setminus (V(\mathcal{P}_0)\cup V(\hc))$, $n_2:=|V_2|$, and $\hh_2:=\hh[V_2]$. Note that $|V(\mathcal{P}_0)\cup V(\hc)|\leq \gamma n+2k(k\gamma)^{19}<2\gamma n$ and thus
	 \[
	 \delta_\ell(\hh_2)\geq \left(\frac{1}{2}-\gamma\right)\binom{n-\ell}{k-\ell} -2\gamma n \binom{n-\ell-1}{k-\ell-1}>  \left(\frac{1}{2}-2k\gamma\right)\binom{n_2-\ell}{k-\ell}.
	 \]
	  We now apply Lemma \ref{lem-pathcover} to find a family of at most $p$ paths $\mathcal{P}_1,\mathcal{P}_2,\ldots,\mathcal{P}_m$ covering all but at most $\alpha |V_2|\leq \alpha n$ vertices in $V_2$ with $\alpha\ll 2k\gamma$. For every $i\in [m]$, let $S_i,T_i$ be  ends of $\mathcal{P}_i$. Due to Lemma \ref{lem-reservoir}, we can connect $S_i$ and $T_{i+1}$, $0\leq i\leq m$ (with $T_{m+1}:=T_0$)  by disjoint $2k$-sets chosen from $\hc$ and get an $(\ell,k-\ell)$-cycle. This is possible because $p+1\leq \gamma^{23}n$.

At last, we use $\mathcal{P}_0$ to absorb all uncovered vertices in $V_2$ and unused vertices in $\hc$. This is possible because the number of uncovered vertices is at most $\alpha n+2k|\hc|\leq \alpha n+ 2k(k\gamma)^{19}n<k\gamma^{18}n$.
\end{proof}

It remains to prove the lemmas. We prove Lemma \ref{lem-reservoir} and Lemma \ref{lem-absorbing} in Section \ref{sec:lem-reser-absor} via a Connecting Lemma, Lemma \ref{lem-connecting}, which itself is proved in Section 2.2. In Section \ref{sec:lem-pathcover} we recall the Weak Regularity Lemma and apply it to prove Lemma \ref{lem-pathcover}.

\subsection{Proofs of the Reservoir Lemma and the Absorbing Lemma.}\label{sec:lem-reser-absor}

Let $\hh$ be a $k$-graph on $n$ vertices and let $\hl:=\binom{V(\hh)}{\ell}$, $\hr:=\binom{V(\hh)}{k-\ell}$. We need the following result, which is often called the Connecting Lemma.

\begin{lem}[Connecting]\label{lem-connecting}
	Given  $\epsilon>0$ and $k\geq 3$, there exist $\gamma>0$ and $n_0\in \hn$ such that the following holds. Suppose $\hh$ is a $k$-graph on $n\geq n_0$ vertices with $\delta_\ell(\hh)\geq(1/2-\gamma)\binom{n-\ell}{k-\ell}$. If $\hh$ is not $\epsilon$-close to $\hb_{n,k}$ or $\overline{\hb}_{n,k}$, then for every disjoint sets $L\in \hl$, $R\in \hr$   there are at least   $\gamma^3 n^{2k}$ connectors for $L$ and $R$.
\end{lem}

\proof[Proof of the Reservoir Lemma ]
 Suppose $1/n\ll\gamma\ll\epsilon\ll1/k$.   By  the Connecting Lemma, for every disjoint $L\in \hl$, $R\in \hr$, there are at least $\gamma^3 n^{2k}$ connectors  for $(L,R)$,  let $\hc_{L,R}$ denote the family of all connectors for $L$ and $R$.

 Set $m=\gamma^{19} n$. Let $\hm$ be an $m$-matching chosen uniformly at random from $\binom{V(\hh)}{2k}$ and  let $\eta(L,R)=|\hm \cap \hc_{L,R}|$. By Theorem \ref{thm-concentrate}, we have
\[
\Pr\left(\left|\eta(L,R)-\frac{\gamma^3 n^{2k}}{\binom{n}{2k}}m\right|>2\gamma^4m\right)<2e^{-\frac{\gamma^8}{2}m} , \quad \mbox{for all disjoint} \quad L\in \hl, R\in \hr.
\]
Since $2e^{-\frac{\gamma^8}{2}m}<\frac{1}{n^{2k}}$ for sufficiently large $n$,  by the union bound, there exists $\hm_0$ such that for all disjoint $L\in \hl$ and $R\in \hr$,
\[
|\hm_0\cap \hc_{L,R}|>\frac{\gamma^3 n^{2k}}{\binom{n}{2k}}m-2\gamma^4m>\gamma^4m=\gamma^{23}n,
\]
the family $\hm_0$ is the desired family $\hc$.
\qed

Next we prove the Absorbing Lemma.
\proof[Proof of Absorbing Lemma]
Given disjoint $L\in \hl$, $R\in \hr$. For an $(\ell,k-\ell)$-path $\mathcal{P}$ on  $10k$ vertices, we say $\hp$ is an $(L,R)$-absorber if there is other $(\ell,k-\ell)$-path on $V(\mathcal{P})\cup L\cup R$, which has the same ends as $\mathcal{P}$. We show that there are many $(L,R)$-absorbers.

\begin{claim}\label{claim-absorbing}
For every disjoint $L\in \hl$, $R\in \hr$,  there are at least $\gamma^{16}n^{10k}$  $(L,R)$-absorbers.
\end{claim}
\proof

By Lemma \ref{lem-connecting}, there are at least $\gamma^3n^{2k}$ connectors for $L,R$.  We first choose two disjoint connectors for $L,R$, denote by $(L_1,R_1,L_2,R_2)$ and $(L_1',R_1',L_2',R_2')$. The number of choices is at least $\gamma^3n^{2k}(\gamma^3n^{2k}-1-3kn^{2k-1})/2=\frac{1}{4}\gamma^6n^{4k}$. Then  we choose three disjoint connector for $(R_1,L_1')$, $(L_2,R_1')$ and $(R_2,L_2')$ respectively from the remaining vertices, denote by $C_1,C_2,C_3$. The number of choices is at least $(\gamma^3n^{2k}-5-11kn^{2k-1})^3\geq \frac{1}{8}\gamma^9n^{6k}$.

 Let $\hp=L_1R_1C_1L_1'R_1'C_2L_2R_2C_3L_2'R_2'$, then $\hp$ is an $(\ell,k-\ell)$-path with ends $(L_1,R_2')$ and note that  $L_1RL_1'C_1R_1L_2C_2R_1'L_2'C_3R_2LR_2'$ is also an $(\ell,k-\ell)$-path with the same ends. Therefore, $\hp$ is an $(L,R)$-absorber. Moreover, the number of such $10k$-sets is at least
\[
\frac{1}{4}\gamma^6n^{4k}\frac{1}{8}\gamma^9n^{6k}\geq \gamma^{16} n^{10k}.
\]
\qed

Let $\ha_{L,R}$ denote the family of all absorbers for $(L,R)$. Then $|\ha_{L,R}|\geq \gamma^{16}n^{10k}$.  Set $m=\gamma^2 n$. Let $\hm$ be a $m$-matching chosen from $\binom{V(\hh)}{10k}$ uniformly at random. Let $\eta(L,R)=|\hm \cap \ha_{L,R}|$, by Theorem \ref{thm-concentrate}  we have
\[
\Pr\left(\left|\eta(L,R)-\frac{\gamma^{16} n^{10k}}{\binom{n}{10k}}m\right|>2\gamma^{17}m\right)<2e^{-\frac{\gamma^{34}}{2}m} , \quad \mbox{for disjoint }  L\in \hl, R\in \hr.
\]
Since $2e^{-\frac{\gamma^{34}}{2}m}<\frac{1}{n^{10k}}$ for sufficiently large $n$,  by the union bound there exists $\hm_0$ such that for  all disjoint   $L\in \hl$, $R\in \hr$, we have
\[
|\hm_0\cap \ha_{L,R}|>\frac{\gamma^{16} n^{10k}}{\binom{n}{10k}}m-2\gamma^{17}m>\gamma^{17}m=\gamma^{19}n.
\]
Then we connect the elements in $\hm_0$ to a single path $\hp$ by using the Connecting Lemma. Since $|V(\hp)|\leq 12km<\gamma n$, this path satisfies the requirements of the lemma.
\qed

\subsection{Proof of Connecting Lemma}
This lemma is a very slight  modification of  Lemma 5.3 in \cite{Treglown2013}.  For the sake of completeness, we include the proof as well.

 Recall that $\hl=\binom{V(\hh)}{\ell}$ and $\hr=\binom{V(\hh)}{k-\ell}$. Set $N:=\binom{n}{\ell}$ and $N':=\binom{n}{k-\ell}$.
Define a bipartite graph $G(\hh)$ on partite sets $\hl$, $\hr$ with the edge set
\[
G(\hh) = \left\{(L,R)\colon L\in \hl, R\in \hr, L\cup R\in \hh\right\}.
\]
When it is clear from the context, we will often refer to $G(\hh)$ as $G$. Moreover, for any $\gamma>0$,  $\delta_\ell(\hh)>\left(\frac 12-\frac{\gamma}{3}\right)\binom{n-\ell}{k-\ell}$ as $n$ is sufficiently large. Thus,
\begin{align}\label{L-degree}
\deg_G(L)>\left(\frac 12-\frac{\gamma}{2}\right)N' \mbox{ for } L\in \hl.
\end{align}
 By proposition \ref{prop-2.1} and  $k-\ell\leq \ell$, $\delta_{k-\ell}(\hh)>\left(\frac 12-\frac{\gamma}{3}\right)\binom{n-k+\ell}{\ell}$, thus
 \begin{align}\label{R-degree}
 \deg_G(R)>\left(\frac 12-\frac{\gamma}{2}\right)N  \mbox{ for } R\in \hr.
 \end{align}

Let $n\geq k\geq 3$. Denote by $B_{n,k}$ the bipartite graph on partite sets $\hl$, $\hr$ satisfying (i), (ii), (iii).
\begin{itemize}
	\item[(i)] $X_1,X_2$ is a partition of $\hl$ such that $|X_1|=\lceil N/2\rceil$ and $|X_2|=\lfloor N/2\rfloor$;
	\item[(ii)] $Y_1,Y_2$ is a partition of $\hr$ such that $|Y_1|=\lceil N'/2\rceil$ and $|Y_2|=\lfloor N'/2\rfloor$;
	\item[(iii)] $B_{n,k}[X_1,Y_1]$ and $B_{n,k}[X_2,Y_2]$ are complete bipartite graphs. Furthermore, there is no other edges in $B_{n,k}$.
\end{itemize}

\begin{lem}[Lemma 5.4, \cite{Treglown2013}]
	Given  $\epsilon>0$ and  $k\geq 3$, there exist $\beta>0$ and $n_0\in \hn$ such that the following holds. Suppose that $\hh$ is a $k$-graph on $n\geq n_0$ vertices. If $\hh$ is not $\epsilon$-close to $\hb_{n,k}$ or $\overline{\hb}_{n,k}$, then $G:=G(\hh)$ is not $\beta$-close to  $B_{n,k}$.\qed
\end{lem}

   Given $\beta>0$, we choose additional constants $\gamma$ such that $0<\gamma \ll   \beta$.

\begin{lem}[Claim 5.6,\cite{Treglown2013}]\label{lem-support}
	 If $G$ is not $\beta$-close to $ B_{n,k}$, then either  (i) or  (ii) holds.

(i) For each $L\in \hl$, the number of $L'\in \hl$ such that
$|\hhn(L)\cap \hhn(L')|\geq \gamma N'$ is at least $(1/2+\gamma)N$.

(ii) The number of $R\in\hr$ such that $|\hhn(R)|\geq (1/2+\gamma)N$ is at least $2\gamma N'$.
\end{lem}

\begin{claim}[Claim 5.5,\cite{Treglown2013}]
	If  Lemma  \ref{lem-support}  (ii) holds, then  for every disjoint sets $L\in \hl$, $R\in \hr$   there are at least   $\gamma^3 n^{2k}$ connectors for $(L,R)$.
\end{claim}
\begin{claim}
	If  Lemma \ref{lem-support} (i) holds, then  for every disjoint sets $L\in \hl$, $R\in \hr$   there are at least   $\gamma^3 n^{2k}$ connectors for $(L,R)$.
\end{claim}

\begin{proof}
Let
\[
\hl'=\{L'\in \hl\colon |\hhn(L)\cap \hhn(L')|\geq \gamma N'\}.
\]
By Lemma \ref{lem-support} (i), $|\hl'|\geq (1/2+\gamma)N$.

We first choose $L_1\in \hhn(R)$ disjoint from $L\cup R$, the  number of choices is at least $\left(\frac12-\frac{\gamma}{2}\right)N-k\binom{n}{\ell-1}\geq \frac{1}{3}N$. Then choose $R_1\in \hhn(L_1)$ disjoint from $L\cup R\cup L_1$ and  the  number of choices is at least $ \frac{1}{3}N'$. Then choose $L_2\in \hhn(R_1)\cap \hl'$ disjoint from $L\cup R\cup L_1\cup R_1$, the  number of choices is at least
\[
|\hhn(R_1)|+|\hl'|-|\hhn(R_1)\cup \hl'|- 2k\binom{n}{\ell-1} \overset{\eqref{R-degree}}{\geq} \left(\frac12-\frac{\gamma}{2}+\frac12+\gamma-1\right)N-2k\binom{n}{\ell-1}\geq \frac{\gamma}{3}N.
\]
Finally, we choose $R_2\in \hhn(L)\cap \hhn(L_2)$ disjoint from $L\cup R\cup L_1\cup R_1\cup L_2$, the number of choices is at least
\[
\gamma N'-3k\binom{n}{k-\ell-1}\geq \frac{\gamma}{2}N'.
\]
Therefore, such an $L_1\cup R_1\cup L_2\cup R_2$ is a connector of $(L,R)$. The total number is at least
\[
\frac{1}{3}N\frac{1}{3}N'\frac{\gamma}{3}N\frac{\gamma}{2}N'\geq \gamma^3 n^{2k}.
\]
\end{proof}

Now the Connecting Lemma immediately follows from  Lemma \ref{lem-support} and Claims 2.3 and 2.4.

\subsection{Proof of Path Cover Lemma}\label{sec:lem-pathcover}
We follow a similar approach in \cite{Han2022}, which uses the weak regularity lemma for hypergraphs, a straghtforward extension of Szemer\'{e}di's regularity lemma for graphs. For $n$  sufficiently large,  the weak regularity lemma guarantees a regularity partition of $\hh$, and a cluster hypergraph can be defined on the partition such that every edge  is a regular $k$-partite $k$-graph.  There are two main steps to prove Lemma \ref{lem-pathcover}. First we get an almost perfect matching of the cluster hypergraph. Then by the regularity,  we could find a tight path covering almost all vertices  for each $k$-partite $k$-graph, which
is an edge of the almost perfect matching. Such paths are the desired ones in Lemma \ref{lem-pathcover}.

Let $\hh$ be a $k$-graph and let $A_1,\dots,A_k$ be pairwise disjoint non-empty subsets of $V(\hh)$. We define $e(A_1,\dots,A_k)$ to be the number of edges with one vertex in each $A_i, i\in [k]$, and the density of $\hh$ with respect to $(A_1,\dots,A_k)$ as
$$d\left(A_{1}, \ldots, A_{k}\right)=\frac{e\left(A_{1}, \ldots, A_{k}\right)}{\left|A_{1}\right| \cdots\left|A_{k}\right|}$$
Given $\varepsilon, d \geq 0$ and  pairwise disjoint sets $V_{1}, \ldots, V_{k} \subset V$, we say $(V_1,V_2,\ldots,V_k)$ is {\it $(\varepsilon, d)$-regular} if
$$
\left|d\left(A_{1}, \ldots, A_{k}\right)-d\right| \leq \varepsilon
$$
for all  $A_{i} \subset V_{i}$ with $\left|A_{i}\right| \geq \varepsilon\left|V_{i}\right|$, $i =1,2,\ldots,k$. We say $\left(V_{1}, \ldots, V_{k}\right)$ is $\varepsilon$-regular if it is $(\varepsilon, d)$-regular for some $d \geq 0$. It is immediate from the definition that for an $(\varepsilon, d)$-regular $k$-tuple $\left(V_{1}, \ldots, V_{k}\right)$, if $V_{i}^{\prime} \subset V_{i}$ has size $\left|V_{i}^{\prime}\right| \geq c\left|V_{i}\right|$ for some $c \geq \varepsilon$, then $\left(V_{1}^{\prime}, \ldots, V_{k}^{\prime}\right)$ is $(\varepsilon / c, d)$-regular.

\begin{thm}\label{thm-regularity}
  For all $t_{0} \geq 0$ and $\varepsilon>0$, there exist $T_{0}=T_{0}\left(t_{0}, \varepsilon\right)$ and $n_{0}=n_{0}\left(t_{0}, \varepsilon\right)$ so that for every $k$-graph $\mathcal{H}$ on $n>n_{0}$ vertices, there exists a partition $V=V_{0} \dot{\cup} V_{1} \dot{\cup} \cdots \dot{\cup} V_{t}$ such that

(i) $t_{0} \leq t \leq T_{0}$,

(ii) $\left|V_{1}\right|=\left|V_{2}\right|=\cdots=\left|V_{t}\right|$ and $\left|V_{0}\right| \leq \varepsilon n$,

(iii) for all but at most $\varepsilon{t\choose k}$ sets $\left\{i_{1}, \ldots, i_{k}\right\} \in{[t]\choose k}$, the $k$-tuple $\left(V_{i_{1}}, \ldots, V_{i_{k}}\right)$ is $\varepsilon$-regular.
\end{thm}
A partition as given in Theorem \ref{thm-regularity} is called an $(\varepsilon, t)$-regular partition of $\mathcal{H}$. For an $(\varepsilon, t)$-regular partition of $\mathcal{H}$ and $d \geq 0$ we refer to $\mathcal{Q}=\left(V_{i}\right)_{i \in[t]}$ as the family of clusters and define the cluster hypergraph $\mathcal{K}=\mathcal{K}(\varepsilon, d, \mathcal{Q})$ with vertex set $[t]$ and $\left\{i_{1}, \ldots, i_{k}\right\} \in\binom{[t]}{k}$ is an edge if and only if $\left(V_{i_{1}}, \ldots, V_{i_{k}}\right)$ is $\varepsilon$-regular and $d\left(V_{i_{1}}, \ldots, V_{i_{k}}\right) \geq d$.

The following corollary shows that the cluster hypergraph inherits the minimum degree property of the original hypergraph. The proof is standard and very similar to that of \cite[Proposition 16]{Han2010}  so we omit the proof.

\begin{cor}\label{cor-2.11}
   For $c, \varepsilon, d>0$,  integers $k \geq 3$, $k/2\leq\ell\leq k-1$ and $t_{0} \geq 2 k^{2} / d$, there exist $T_{0}$ and $n_{0}$ such that the following holds. Given a $k$-graph $\mathcal{H}$ on $n>n_{0}$ vertices with $\delta_{\ell}(\mathcal{H}) \geq c\binom{n-\ell}{k-\ell}$, there exists an $\varepsilon$-regular partition $\mathcal{Q}=\left(V_{i}\right)_{i \in[t]}$, with $t_{0} \leq t \leq T_{0}$. Furthermore, let $\mathcal{K}=\mathcal{K}(\varepsilon, d / 2, \mathcal{Q})$ be the cluster hypergraph of $\mathcal{H}$. Then the number of $\ell$-sets  $S \in\binom{[t]}{\ell}$ violating $\operatorname{deg}_{\mathcal{K}}(S) \geq(c-\sqrt{\varepsilon}-d)\binom{t-\ell}{k-\ell}$ is at most $\sqrt{\varepsilon}\binom{t}{\ell}$.
\end{cor}

We use the following proposition from \cite[Claim 4.1]{Roedl2008}.

\begin{prop}[Claim 4.1,\cite{Roedl2008}]\label{prop-2.12}
  Given $c>0$ and $k\geq 2$, every $k$-partite $k$-graph $\hh$ with at most $m$ vertices in each part and with at least $cm^k$ edges contains a tight path on at least $cm$ vertices.
\end{prop}

We shall use Proposition \ref{prop-2.12} to cover an $(\epsilon,d)$-regular tuple $(V_1,\dots,V_k)$ by $(\ell,k-\ell)$-paths.

\begin{lem}\label{lem-2.13}
  Given $k\geq 3$ and $\epsilon,d>0$ such that $d>2\epsilon$. Let $m>\frac{k}{\epsilon(d-\epsilon)}$. Suppose that $(V_1,V_2,\dots,V_k)$ is an $(\epsilon,d)$-regular $k$-tuple with $|V_i|=m$ for $i\in [k]$. Then there is a family consisting of $\frac{k}{(d-2\epsilon)\epsilon}$ pairwise vertex-disjoint $(\ell,k-\ell)$-paths which cover all but at most $k\epsilon m$ vertices of $V_1\cup V_2\cup \dots\cup V_k$.
\end{lem}
\proof We greedily find paths by Proposition \ref{prop-2.12} in $V_1\cup V_2\cup \dots\cup V_k$ until every cluster has less than $\epsilon m$ vertices uncovered. Assume that every cluster has $m'\geq \epsilon m$ vertices uncovered. By regularity, the remaining hypergraph has at least $(d-\epsilon)(m')^k$ edges. We apply Proposition \ref{prop-2.12} and get a path covering at least $(d-\epsilon)m'\geq (d-\epsilon)\epsilon m$ vertices. Thus, the number of paths is at most $km/((d-\epsilon)\epsilon m)=\frac{k}{(d-\epsilon)\epsilon }$.\qed

We will find an almost perfect matching in the cluster hypergraph.
\begin{thm}[\hspace{1sp}\cite{Han2022}]\label{thm-2.14}
  For each integer $k\geq 3$, $1\leq d\leq k-2$ and every $0<\gamma<1/4$, $\epsilon>0$ the following holds for sufficiently large $n$. Suppose that $\hh$ is a $k$-graph on $n$ vertices such that for all but at most $\epsilon\binom{n}{d}$ $d$-sets $S$,
  \[
  \deg(S)\geq \left(\frac{k-d}{k}-\frac{1}{k^{k-d}}+\gamma\right)\binom{n-d}{k-d}.
  \]
  Then $\hh$ contains a matching that covers all but at most $2\epsilon^{1/k}n$ vertices.
\end{thm}
Now we are ready to prove Lemma \ref{lem-pathcover}.
\proof[Proof of Lemma \ref{lem-pathcover}.] Let $k,\ell$ be integers such that  $k\geq 3$ and $\frac k2\leq  \ell\leq k-1$. Suppose $1/n\ll 1/p \ll 1/T_0 \ll 1/t_0 \ll \epsilon \ll \alpha \ll \gamma \ll 1/k$.

 Suppose $\hh$ is a $k$-graph on $n$ vertices and $\delta_\ell(\hh)\geq (\frac 12-\gamma)\binom{n-\ell}{k-\ell}$. We apply Corollary \ref{cor-2.11} with parameters $\frac 12-\gamma,\epsilon,2\gamma$ and $t_0$ obtaining an $(\epsilon,t)$-regular partition $\mathcal{Q}=(V_i)_{i\in [t]}$ with $t_0\leq t\leq T_0$ and the cluster hypergraph $\mathcal{K}=\mathcal{K}(\epsilon,\gamma,\mathcal{Q})$ with vertex set $[t]$. Let $m\geq\frac{(1-\epsilon)n}{t}$ be the size of each cluster $V_i, i\in[t]$. By Corollary \ref{cor-2.11}, for all but at most $\sqrt{\epsilon}\binom{t}{\ell}$ $\ell$-sets $S$,
\[
\deg_\mathcal{K}(S)\geq\left(\frac 12-\gamma-\sqrt{\epsilon}-2\gamma\right)\binom{t-\ell}{k-\ell}\geq\left(\frac 12-4\gamma\right)\binom{t-\ell}{k-\ell}.
\]
Note that we have $\frac 12-4\gamma>\frac{\ell}{k}-\frac{1}{k^{k-\ell}}+\gamma$ because $\gamma$ is small. Thus by Theorem \ref{thm-2.14}, $\mathcal{K}$ contains a matching $\mathcal{M}$ covering all but at most $2\epsilon^{1/k}t$ vertices. For each edge $\{i_1,\dots,i_k\}\in \mathcal{M}$, the corresponding clusters $(V_{i_1},\dots,V_{i_k})$ is $(\epsilon,\gamma')$-regular for some $\gamma'\geq \gamma$. Thus we can apply Lemma \ref{lem-2.13} on $(V_{i_1},\dots,V_{i_k})$ and get a family of at most $\frac tk\cdot\frac{2k}{(\gamma-2\epsilon)\epsilon}\leq p$ paths, which leaves at most
\[
|V_0|+k\epsilon m\cdot\frac tk+2\epsilon^{1/k}t\cdot m\leq \epsilon n+\epsilon n+2\epsilon^{1/k}n\leq \alpha n
\]
vertices uncovered in $\hh$.\qed

\section{Extremal Case - Proof of Theorem \ref{thm-extr} }\label{sec:extr case}

This section is devoted to the proof of Theorem \ref{thm-extr}.  Let $\hb$ be an $n$-vertex $k$-graph such that $\hb=\hb_{n,k}(A,B)$ or $\ohb_{n,k}(A,B)$ for some $V=A\dot{\cup} B$ with $k|n$. Then for any $\gamma>0$,
\begin{align}\label{eq-12051}
	\delta_\ell(\hb)>(1/2-\gamma)\binom{n-\ell}{k-\ell}
\end{align}
as $n$ sufficiently large. We define $\eta_\hb(S)$ to be  an indicator function  that equals $1$ if  $|S\cap A|$ is odd, $0$ else. Define $\eta(\hb)=1$ if $\hb=\hb_{n,k}(A,B)$ and $\eta(\hb)=0$ if $\hb=\ohb_{n,k}(A,B)$. Note that for any $S\subset A\cup B$, $\hb-S= \hb_{n-|S|,k}(A\setminus S,B\setminus S)$ or  $\ohb_{n-|S|,k}(A\setminus S,B\setminus S)$. It follows that $\eta(\hb)=\eta(\hb-S)$ for all $S\subset A\cup B$. Define
\[
f(\hb) = \eta(\hb)\frac{n}{k}+|A| \pmod 2.
\]
It is easy to check that $\hb \in \hh_{ext}(n,k)$ if and only if $f(\hb)=1$. Moreover, $\hb$ contains a Hamilton $(\ell,k-\ell)$-cycle if and only if $f(\hb)=0$.

For a $k$-set $E\subset V(\hb)$, let $\hb'=\hb-E$, $A'=A\setminus E$ and  $n'=n-k$. Note that $\eta(E)= \eta(\hb)$ implies $E\in \hb$.   We see that $\eta(E)= \eta(\hb)$ also implies that
\begin{align}\label{eq-3.2}
f(\hb')=\eta(\hb')\frac{n'}{k}+|A'|\pmod 2\equiv\eta(\hb)\frac{n-k}{k}+|A|-\eta(E) \equiv \eta(\hb)\frac{n}{k}+|A|  \pmod 2=f(\hb).
\end{align}

Let $\hh$ be a $k$-graph  on  $V(\hb)$ and let $0\leq\alpha\leq 1$. We say $S\subset V$ is  $\alpha$-good in $\hh$ with respect to $\hb$ if $\deg_{\hb\setminus\hh}(S)\leq \alpha \binom{n-|S|}{k-|S|}$.
Moreover, we say $\hh$ is  $\alpha$-{\it good with respect to} $\hb$ if every vertex  is $\alpha$-good. For a set pair  $(L,R)$, we say it is $\alpha$-good if both $L$ and $R$ are $\alpha$-good.

The following properties are useful in  our proofs.

\begin{prop}\label{claim:typical}
     Given real  $0< \epsilon\leq 1$ and  integer $1\leq j\leq k-1$. Let  $\epsilon':= \sqrt{k^k\epsilon} $. Suppose that  $\hh$ is $\epsilon$-close to $\hb$. Then the number of not $\epsilon'$-good $j$-set is at most $\epsilon' n^j$.
\end{prop}
\begin{proof}  Let $m$ be the number of $j$-sets that are not $\epsilon'$-good. Since $\hh$ is $\epsilon$-close to $\hb$, there are at most $\epsilon n^k$ edges in $\hb\setminus \hh$. It follows that
\[
m\epsilon' \left(\frac{n}{k}\right)^{k-j} \leq m \epsilon'\binom{n-j}{k-j}\leq |\hb\setminus \hh|\binom{k}{j}\leq \binom{k}{j}\epsilon n^k.
\]
Then $m\leq k^{k-j}\binom{k}{j}\epsilon n^j/\epsilon'$. By setting $\epsilon'= \sqrt{k^{k}\epsilon}$, we conclude that $m\leq \epsilon' n^j$.
\end{proof}

\begin{prop}[\hspace{1sp}\cite{Treglown2012}]\label{prop:3.2}
 Given reals $0<\alpha'<1$ and $0\leq c< 1$. Let $\alpha:=\alpha'/c^{k-|S|}$. Suppose that  $S$ is $\alpha'$-good in $\hh$ with respect to $\hb$. Let $\hb'$ be a subgraph of $\hb$ on $U\subset V$ such that $S\subset U$ and $|U|\geq cn$. Then $S$ is $\alpha$-good in $\hh[U]$ with respect to $\hb$.
\end{prop}
We show in the following  that for every two $j$-sets if they are ``good'', disjoint and have the same parity, then they   can be connected by many $(k-j)$-sets. The proof is clear as the proposition described.
\begin{prop}\label{prop-3.3}
 Given positive reals $0<\alpha_1,\alpha_2<1$ and integer $1\leq j\leq k-1$. Suppose that $S_i, i=1,2$ is $\alpha_i$-good $j$-set in $\hh$ with respect to $\hb$ and $S_1$, $S_2$ are disjoint with $\eta(S_1)=\eta(S_2)$. Then there are at least
\[
\delta_j(\hb)-\alpha_1\binom{n-j}{k-j}-\alpha_2\binom{n-j}{k-j}-|S_1\cup S_2|\binom{n-1}{k-j-1}
\]
$(k-j)$-sets $T$ disjoint from $S_1,S_2$ and $T\cup S_1, T\cup S_2\in \hh\cap\hb$. i.e., $S_1,S_2$ are connected  by edges in $\hh\cap\hb$.
\end{prop}

 For a $k$-graph $\hh$ which is $\epsilon$-close to $\hb$, if $f(\hb)=1$  then there exists no Hamilton $(\ell,k-\ell)$-cycle in $\hh\cap \hb$. Therefore, we must use at least one  edge in $\hh\cap \ohb$ to find  a Hamilton $(\ell,k-\ell)$-cycle in $\hh$. We call this phenomenon the {\it parity obstacle}.

As in Theorem \ref{thm-extr}, $\hh$ is $\epsilon$-close to  $\hb$ for $|A|,|B|$ almost equal and $\delta_\ell(\hh)>\delta(n,k,\ell)$. The following lemma allow us to find an $(\ell,k-\ell)$-path covering all not $\alpha$-good vertices and settle the parity obstacle.

\begin{lem}\label{lem:main parity}
	Given $k/2\leq \ell \leq k-1$ and $\alpha>0$. There exists $\epsilon>0$ such that the following holds. Suppose that $|A|,|B|$ almost equal, $\hh$ is $\epsilon$-close to  $\hb$  and $\delta_\ell(\hh)>\delta(n,k,\ell)$.  Then there exists an $(\ell,k-\ell)$-path $\mathcal{P}=L\hp^0R$ in $\hh$ such that the following hold: let $V'=V\setminus V(\hp^0)$, $A'=A\setminus V(\hp^0)$,  $\hb'=\hb[V']$,
	 $\hh'=\hh[V']$, then (i) $|V(\hp)|\leq 0.05n$; (ii)  $(L,R)$  is $\alpha$-good  with respect to  $\hb'$ and $L\cup R\in \hb'$ ; (iii) $\hh'$ is $\alpha$-good with respect to $\hb'$; (iv)  $f(\hb')=0$.
\end{lem}

By using Lemma \ref{lem:main parity}, we are left with a ``good" sub-hypergraph which contains no ``bad" vertices and does not have the parity obstacle. The next lemma allow us  to find a Hamilton $(\ell,k-\ell)$-path  covering all the remaining vertices.

\begin{lem}\label{lem-stability}
	Given  $\ell,k$ with $1\leq \ell\leq k-1$ and $k\geq 7$. There exist real $\alpha>0$ and integer $n_0$  such that the following hold: Let $\hb=\hb_{n,k}(A,B)$ or $\ohb_{n,k}(A,B)$ with $n\geq n_0$, $|A|,|B|\geq 0.45n$ and $f(\hb)=0$. If $\hg$ is $\alpha$-good with respect to $\hb$, then for any $\alpha$-good set pair $(L,R)$ with $L\cup R\in \hb$,  $\hg$ contains a Hamilton $(\ell,k-\ell)$-path with ends  $L$ and $R$.
\end{lem}
\begin{proof}[Proof of Theorem \ref{thm-extr}]
Let $\hh$ be $\epsilon$-close to  $\hb_{n,k}(A,B)$ or $\ohb_{n,k}(A,B)$ with $||A|-|B||\leq 1$ and $\delta_\ell(\hh)> \delta(n,k,\ell)$. By  applying Lemma \ref{lem:main parity}, we find a path $\hp=L\hp^0R$  in $\hh$ such that (i), (ii), (iii) and (iv) in Lemma \ref{lem:main parity} hold. Let $ \hg=\hh'$ and let $A'=A\setminus V(\hp^0)$,  $B'=B\setminus V(\hp^0)$. By Lemma \ref{lem:main parity} (iii),  $\hg$ is $\alpha$-good with respect to $\hb'$. By Lemma \ref{lem:main parity} (i) we infer that $|A'|,|B'|\geq  0.45n\geq 0.45|V'|$. By Lemma \ref{lem-stability}, there exists a Hamilton $(\ell,k-\ell)$-path $L\mathcal{Q}R$ in $\hg$. Then $L\hp^0R\mathcal{Q}L$  is a Hamilton $(\ell,k-\ell)$-cycle of $\hh$.
\end{proof}

\subsection{Proof of lemma \ref{lem:main parity}}

Recall that we are in the setting that $\hh$ is $\epsilon$-close to  $\hb$ for $|A|,|B|$ almost equal and $\delta_\ell(\hh)>\delta(n,k,\ell)$. By Proposition \ref{claim:typical}, there are at most $\epsilon'n$ vertices not $\epsilon'$-good with respect to $\hb$. Let $V_0$ denote the family of these vertices. Let $V_0'\subseteq V_0$ be  the set of  all the vertices that are not $1/4$-good with respect to $\hb$. We move the vertices in $V_0'\cap A$ and  $V_0'\cap B$ to the other part, get a new partition $V=A_1\cup B_1$,  i.e.,
\[
A_1:=(A \setminus V_0')\cup (B\cap V_0'),\quad  B_1:=(B \setminus V_0')\cup (A\cap V_0').
\]
Let $\hb_1=\hb(A_1,B_1)$. That is, if $\hb= \hb_{n,k}(A,B)$ then $\hb_1= \hb_{n,k}(A_1,B_1)$ and if $\hb= \ohb_{n,k}(A,B)$ then $\hb_1= \ohb_{n,k}(A_1,B_1)$.
Since  each vertex is contained in  at most $|V_0'|\binom{n-2}{k-2}\leq \epsilon'\binom{n-1}{k-1}$ edges that intersect $V_0'$. Then  each vertex is $(1/4+\epsilon')$-good with respect to $\hb_1$. Similarly, each vertex in $V\setminus V_0$ is $\epsilon'+\epsilon'=2\epsilon'$-good with respect to $\hb_1$. Moreover, since
\[
|\hb_1\setminus\hh|\leq |V_0'|\binom{n-1}{k-1}+\epsilon n^k<\epsilon'n^k,
\]
 $\hh$ is $\epsilon'$-close to $\hb_1$.

 The following claim allows us to find a path in $\hh\cap \hb_1$ to cover any non-empty  $M\subseteq V_0$. This claim is a simple modification of \cite[Lemma 3.5]{Han2015a}. For completeness, let us include the full proof.

 \begin{claim}\label{claim:3.1}
  	Let $M\subset V_0$ and let $U$ be an arbitrary vertex set of size at most $2k$. Then  there exists  an $(\ell,k-\ell)$-path $\hp'$ with ends $R'$, $R\in \hr$  such that
  (i) $V(\hp')$ is a subset  of $V\setminus U$ with size $2k|M|-\ell$ covering $M$;  (ii) $\hp'$  contains only edges in $\hh\cap\hb_1$;  (iii) the ends $R, R'$  are both $\epsilon''$-good  with respect to $\hb_1$ where $\epsilon''=\sqrt{2k^k\epsilon'}$; (iv) $\eta(R)=\eta(R')$ and their values   equal 0 or 1 depending on our choice.
  \end{claim}
  \begin{proof}
  	For each vertex $v\in M$, let $\hl_v$ be the $(k-1)$-graph consisting of all $(k-1)$-sets $E$ that satisfy the following:
  	\begin{itemize}
  		\item[(a)] $E\cup\{v\}\in\hh\cap \hb_1$;
  		\item[(b)] $E\cap (M\cup U)=\emptyset$;
  		\item[(c)] all $(k-\ell)$-subsets of $E$ are $\epsilon''$-good with respect to $\hb_1$.
  	\end{itemize}
 Since each vertex in $V(\hh)$ is $(\frac{1}{4}+\epsilon')$-good with respect to $\hb_1$, by \eqref{eq-12051} there are at most
 \[
 \left(\frac{1}{2}-\gamma-\frac{1}{4}-\epsilon'\right)\binom{n-1}{k-1}-(|M|+|U|)\binom{n-2}{k-2}> \left(\frac{1}{4}-2\epsilon'\right)\binom{n-1}{k-1}
 \]
$(k-1)$-sets satisfying (a) and (b). Since $\hh$ is $\epsilon'$-close to $\hb_1$, by Property \ref{claim:typical}  and letting $\epsilon''=\sqrt{k^k\epsilon'}$ there are at most $\epsilon''n^{k-\ell}$ $(k-\ell)$-sets are not $\epsilon''$-good. Then the number of $(k-1)$-sets that does not satisfy (c) is at most
$
\epsilon''n^{k-\ell}\cdot\binom{n}{\ell-1}<k^k\epsilon''\binom{n-1}{k-1}.
$
We infer that
\[
|\hl_{v}|\geq (\frac{1}{4}-2\epsilon'-k^k\epsilon'')\binom{n-1}{k-1}>\frac{1}{5}\binom{n-1}{k-1}.
\]

 Assume that $M=\{v_1,v_2,\ldots,v_t\}$. For  $i=1,2\ldots,t$, we greedily find $E_1^i,E_2^i\in \hl_{v_i}$ such that $|E_1^i\cap E_2^i|=\ell-1$ and $\eta(E_1^i\setminus E_2^i)$, $\eta(E_2^i\setminus E_1^i)$ are both 0 (or 1, depended on our choice). This is possible since in the $i$-th step there are at most $(2k-\ell)|M|\leq 2k\epsilon' n$ vertices  used in $E_1^j\cup E_2^j$ for $j<i$. Then at least
 \begin{align}\label{ineq-new3.1}
 \frac{1}{5}\binom{n-1}{k-1}-2k\epsilon'n\binom{n-2}{k-2}\geq \frac{1}{6}\binom{n-1}{k-1}
 \end{align}
 edges in $\hl_{v_i}$ that do not intersect $\cup_{j=1}^{i-1} (E_1^j\cup E_2^j\cup \{v_j\})$. Among these edges we claim that there exist  $E_1^i,E_2^i$  sharing $\ell-1$ vertices   such that $\eta(E_1^i\setminus E_2^i)$, $\eta(E_2^i\setminus E_1^i)$ are both 0 (or 1). To the contrary, suppose that no such sets  $E_1^i,E_2^i$ exist.  For each $E\in L_{v_i}$ that is not contained in $A_1$ or $B_1$,  we can arbitrarily partition it into $S_i\cup T_i$ such that $|S_i|=\ell-1 $ and $\eta(T_i)$ is 0 (or 1).  Let $\hl_{v_i}(S_i)$ be the family of all $(k-\ell)$ sets $R$ such that $S_i\cup R\in \hl_{v_i}$. Note that $S_i\cup R$ and $S_i\cup T_i\in \hl_{v_i}$ imply $\eta(R)=\eta(T_i)$. By our assumption, $\hl_{v_i}(S_i)$ must be intersecting, and thus by Erd\H{o}s-Ko-Rado theorem \cite{ERDOS1961} it has size at most $\binom{n-\ell-1}{k-\ell-1}$. Since there are at most $\binom{n-1}{\ell-1}$ choices for $S_i$, and at most $\binom{|A_1|}{k-1}+\binom{|B_1|}{k-1}$ $(k-1)$-sets that are contained in $A_1$ or $B_1$,   we infer that  the number of edges in $\hl_{v_i}$ that do not intersect $\cup_{j=1}^{i-1} (E_1^j\cup E_2^j\cup \{v_j\})$ is at most
  	\[
  \binom{n-1}{\ell-1}\binom{n-\ell-1}{k-\ell-1}+\binom{|A_1|}{k-1}+\binom{|B_1|}{k-1}.
  	\]
  	As $k\geq 5$, $|A_1|,|B_1|\leq \frac{n}{2}+\epsilon'n$ and  $n$ is sufficiently large, it is at most $\frac{1}{6}\binom{n-1}{k-1}$, contradicting \eqref{ineq-new3.1}.
  	
  	It remains to connect these short paths to a single path. For  $i=1,2\ldots,t$, we greedily connect $E_2^i\setminus E_1^i$ and $E_1^{i+1}\setminus E_2^{i+1}$ by  $\ell$-set in the remaining vertices. This is possible since  $E_2^i\setminus E_1^i$, $E_1^{i+1}\setminus E_2^{i+1}$ are   $\epsilon''$-good with respect to $\hb_1$ and have the same parity. By \eqref{eq-12051},  $|V_0|\leq \epsilon'n$  and  Proposition \ref{prop-3.3},  there are at least
  	\begin{align*}
  \delta_{k-\ell}(\hb_1)-2\epsilon''\binom{n-k+\ell}{\ell}-|\cup_{j=1}^{t} (E_1^j\cup E_2^j\cup \{v_j\})|\binom{n-1}{\ell-1}
   &\geq \delta_{k-\ell}(\hb_1)-2\epsilon''\binom{n-k+\ell}{\ell}-2k|V_0|\binom{n-1}{\ell-1}\\[5pt]
   &>0.1\binom{n-k+\ell}{\ell}
  	\end{align*}
  	$\ell$-sets  connecting $E_2^i\setminus E_1^i$, $E_1^{i+1}\setminus E_2^{i+1}$  in the remaining vertices.  Then we obtain an $(\ell,k-\ell)$-path $\hp'$  with ends  $E_1^1\setminus E_2^1$, $E_2^{t}\setminus E_1^{t}$. Clearly, $E_1^1\setminus E_2^1$, $E_2^{t}\setminus E_1^{t}$ have the same parity.
  \end{proof}

We prove Lemma \ref{lem:main parity}  by distinguishing two cases.

{\bf Case 1.} $f(\hb_1)=0$.

Let $M=V_0$ and let $\hp'=R'\hp^0R$ be the $(\ell,k-\ell)$-path with ends $R', R$ obtained by  applying Claim \ref{claim:3.1}. Since $\hh$ is $\epsilon'$-close to $\hb_1$,  by Proposition \ref{claim:typical}  we can find an   $\ell$-set $L$ in $V\setminus V(\hp')$ such that $L$ is $\epsilon''$-good with respect to $\hb_1$ and $LR'\in \hh \cap \hb_1$. Then $\hp=LR'\hp^0R$ is an $(\ell,k-\ell)$-path with ends $L, R$. We verify that $\hp$ satisfy the requirements in Lemma \ref{lem:main parity}.  Let $V'=V\setminus V(R'\hp^0)$, $A'=A\setminus V(R'\hp^0)$, $B'=B\setminus V(R'\hp^0)$, $\hh'=\hh[V']$, $\hb'=\hb_1[V']$. Then

(i) $|V(\hp)|\leq 2k\epsilon'n\leq 0.05n$.

(ii)  Since $\eta(R')= \eta(R)$ and $LR'\in  \hb_1$,    $LR\in\hb_1$. Therefore $LR\in\hb'$. Recall that $L,R$ are $\epsilon''$-good with respect to $\hb_1$. By Proposition \ref{prop:3.3}, $L,R$ are $\frac{\epsilon''}{0.9^{k}}$-good with respect to $\hb'$. Therefore, $(L,R)$ is a $\frac{\epsilon''}{0.9^{k}}$-good set pair.

(iii) Since each vertex in $V'$ is $2\epsilon'$-good with respect to $\hb_1$ and $|V(\hp)|< 0.05n$,  by Proposition \ref{prop:3.2}, all vertex in $V'$ are $\frac{2\epsilon'}{0.9^{k-1}}$-good with respect to $\hb'$.

(iv)
Since $\hp$ can be viewed as a matching so that every element is an edge of $\hb_1$. By \eqref{eq-3.2}, we see that
\[
f(\hb')=\eta(\hb')\cdot \frac{|V'|}{k}+|A'| \pmod 2\equiv \eta(\hb_1)\frac{n}{k}+|A_1| \pmod 2=f(\hb_1)=0.
\]

{\bf Case 2.} $f(\hb_1)=1$.

We need the following Lemma due to H\'{a}n, Han and Zhao.

\begin{lem}[\hspace{1sp}\cite{Han2015a}]\label{lem:3.3}
	Given an even integer $k\geq 3$ and an integer $\ell$ with $ k/2\leq \ell\leq k-1$. Let $n$ be sufficiently large. Suppose that $\hh$ is an $n$-vertex $k$-graph with a partition $V(\hh)=A_1\cup B_1$ such that $|A_1|,|B_1|\geq 0.45n$ and $\delta_\ell(\hh)>\delta_\ell(\hb_{n,k}(A_1,B_1))$ (respectively, $\delta_\ell(\hh)>\delta_\ell(\overline{\hb}_{n,k}(A_1,B_1))$). Then $\hh\cap \overline{\hb}_{n,k}(A_1,B_1)$ (respectively, $\hh\cap \hb_{n,k}(A_1,B_1)$) contains two edges $e_1,e_2$ such that $|e_1\cap e_2|\in \{0,\ell\}$.
\end{lem}

{\bf Case 2.1.} If there exists an  $\ell$-set or $(k-\ell)$-set that is not $1/5$-good with respect to $\hb_1$.

If there exists a $(k-\ell)$-set that is not $1/5$-good with respect to $\hb_1$, then we claim that there is an  $\ell$-set that is not $1/5$-good as well. Indeed, otherwise assume that every $\ell$-set is $1/5$-good. Since $\ell\geq k-\ell$, each  $(k-\ell)$-set is contained in $\binom{n-(k-\ell)}{\ell-(k-\ell)}$ $\ell$-sets. Then for every   $(k-\ell)$-set $R$, the number of edges in $\ohb_1\cap \hh$ containing $R$ is at most
\[
\binom{n-(k-\ell)}{\ell-(k-\ell)}\cdot\frac{1}{5}\binom{n-\ell}{k-\ell}/\binom{k-(k-\ell)}{\ell-(k-\ell)}=\frac 15 \binom{n-k+\ell}{\ell},
\]
which contract the assumption that there exists a $(k-\ell)$-set that is not $1/5$-good with respect to $\hb_1$.

Suppose that $L^*$ is an  $\ell$-set that is not $1/5$-good with respect to $\hb_1$. By \eqref{eq-12051} and Proposition \ref{claim:typical}(i),  there  are at least $\frac{1}{5}\binom{n-\ell}{k-\ell}-\epsilon''n^{k-\ell}>\frac{1}{6}\binom{n-\ell}{k-\ell}$ $(k-\ell)$-sets in $\mathcal{N}_{\hh\cap \overline{\hb}_1}(L^*)$, which are $\epsilon''$-good with respect to $\hb_1$.
Then  we may choose disjoint $(k-\ell)$-sets $R_1^*,R_2^*$  such that $L^*\cup R_1^*, L^*\cup R_2^*\in \hh\cap \overline{\hb}_1$, and $R_1^*,R_2^*$ are both $\epsilon''$-good with respect to $\hb_1$.

Let $M=V_0\setminus (R_1^*\cup R_2^*\cup L^*)$, and let  $\hp'=R'\hp^0R$ be the $(\ell,k-\ell)$-path with ends $R$, $R'$ by applying Claim \ref{claim:3.1}. Since the values of $\eta(R)=\eta(R')$ equal $0$ or $1$ depending on our choice, we let  $\eta(R)=\eta(R')=\eta(R_1^*)=\eta(R_2^*)$. By Proposition \ref{prop:3.3},  we can connect $R'$ and $R_2^*$ by at least
\[
\delta_{k-\ell}(\hb_1)-\frac{2}{5}\binom{n-k+\ell}{\ell}-2k|V_0|\binom{n-1}{\ell-1}\geq0.04\binom{n-k+\ell}{\ell}
\]
$\ell$-sets. Choose one of them and  denote it by $L_2$,  then $R_2^*L_2,L_2R'\in  \hh\cap \hb_1$. Since $R_1^*$ is $1/5$-good, by Proposition \ref{claim:typical} and \eqref{eq-12051},  there exists an $\epsilon''$-good $\ell$-set $L_1\in \mathcal{N}_{\hh\cap \hb_1}(R_1^*)$ from the remaining vertices. Thus we get an $(\ell,k-\ell)$-path
\[
L_1R_1^*L^*R_2^*L_2R'\hp^0R.
\]
Let $\hp^1=R_1^*L^*R_2^*L_2R'\hp^0$, $V'=V\setminus V(\hp^1)$, $A'=A\setminus V(\hp^1)$, $B'=B\setminus V(\hp^1)$, $\hh'=\hh[V']$, $\hb'=\hb_1[V']$.
It is clear that such a path  satisfies  (i)(iii) of Lemma \ref{lem:main parity}. Note that $L_1\cup R_1^*\in \hb_1$ and $L^*\cup R_1^*\in  \overline{\hb}_1$ imply   $\eta(L_1)\neq \eta(L^*)$. Similarly $\eta(L^*)\neq \eta(L_2)$. It follows that  $\eta(L_1)=\eta(L_2)$ and therefore $L_1\cup R'\in \hb_1$. Since $\eta(R')=\eta(R)$, we see $L_1\cup R\in \hb_1$ and (ii) holds. Since  $L_1R_1^*L^*R_2^*L_2R'\hp^0R$ can be viewed as a matching $\{L_1\cup R_1^*, L^*\cup R_2^*,L_2\cup R',\ldots\}$, which contains exactly one edge $L^*\cup R_2^*$ in $\ohb_1$. By \eqref{eq-3.2} we infer  (iv) holds.

{\bf Case 2.2.} If all the  $\ell$-sets and $(k-\ell)$-sets are $1/5$-good with respect to $\hb_1$.

By Lemma \ref{lem:3.3}, we get two edges $e_1,e_2\in\hh\cap\ohb_1$ either (i) $|e_1\cap e_2|=\ell$ and $e_1\setminus e_2,e_2\setminus e_1$ are $1/5$-good, or (ii) $|e_1\cap e_2|=0$, every $(k-\ell)$-subsets and $\ell$-subsets of $e_1,e_2$ are $1/5$-good. The case (i) is same as Case 2.1. So we only consider the case (ii).

Let  $e_1=L_1^*\dot{\cup} R_1^*$ and $e_2=L_2^*\dot{\cup} R_2^*$ be a partition such that $\eta(R_1^*)=\eta(R_2^*)$ and $\eta(L_1^*)=\eta(L_2^*)$. Such a partition exists because $\eta(e_1)=\eta(e_2)$.   Let $M=V_0\setminus (e_1\cup e_2)$, and let  $\hp'=R'\hp^0R$ be the $(\ell,k-\ell)$-path with ends $R$ and $R'$ obtained by applying Claim \ref{claim:3.1}. Since the value of $\eta(R)$ can be decided depending on our choice, we choose $R$ such that    $\eta(R)=\eta(R')=\eta(R_1^*)=\eta(R_2^*)$ when applying Claim \ref{claim:3.1}. Since $L_1^*,L_2^*$ are $1/5$-good,  we can connect $L_1^*,L_2^*$ by a $(k-\ell)$-set $R_3$ from the remaining vertex such that $L_1^*R_3,R_3L_2^*\in \hh\cap \hb_1$. Similarly, connect $R_2^*,R'$ by an $\ell$-set $L_3$.  Since $R_1^*$ is $1/5$-good, by Proposition \ref{claim:typical} and \eqref{eq-12051},  there exists an $\epsilon''$-good $\ell$-set $L_4$ from the remaining vertex to connect $R_1^*$, thus we get an $(\ell,k-\ell)$-path
\[
L_4R_1^*L_1^*R_3L_2^*R_2^*L_3R'\hp^0R.
\]
Let $\hp^1=R_1^*L_1^*R_3L_2^*R_2^*L_3R'\hp^0$, $V'=V\setminus V(\hp^1)$, $A'=A\setminus V(\hp^1)$, $B'=B\setminus V(\hp^1)$, $\hh'=\hh[V']$, $\hb'=\hb_1[V']$.
Clearly  such  a path  satisfies  (i)(iii) of Lemma \ref{lem:main parity}. Since $e_1,e_2\in \ohb_1$,  $\eta(L_4)\neq \eta(L_1^*)$, $\eta(L_2^*)\neq\eta(L_3)$ and $\eta(L_1^*)=\eta(L_2^*)$.  It implies that $\eta(L_4)=\eta(L_3)$ and thereby $L_4\cup R'\in \hb_1$. Since $\eta(R')=\eta(R)$, $L_4\cup R\in \hb_1$ and  (ii) holds. Since  $L_4R_1^*L_1^*R_3L_2^*R_2^*L_3R'\hp^0R$ can be viewed as a matching $\{L_4\cup R_1^*, L_1^*\cup R_3, L_2^*\cup R_2^*,L_3\cup R',\ldots\}$, which contains exactly one  edge $L_2^*\cup R_2^*$ in $\ohb_1$. By \eqref{eq-3.2}  we infer (iv) holds .

Recall that  $\epsilon'=\sqrt{k^k\epsilon}$ and $\epsilon''=\sqrt{k^k\epsilon'}$. By letting $\alpha=\frac{\epsilon''}{0.9^{k}}$, we complete the proof of Lemma \ref{lem:main parity}.

\subsection{Proof of Theorem \ref{lem-stability}}

Let  $X_1,\ldots,X_k$ be disjoint  sets with $|X_i|=n$. Let $\mathcal{K}(X_1,X_2,\ldots,X_k)$ be a $k$-graph consists of all the edges that intersect each $X_i$ exactly one vertex, denoted by $\hk$ for brief.  It is easy to see that $\mathcal{K}$ contains a Hamilton $(\ell,k-\ell)$-cycle for arbitrary $1\leq \ell\leq k-1$.

Let $\hf\subseteq \hk$. Suppose that $1\leq j\leq k$ and  $\alpha>0$. We call a  $j$-set $J$ {\it $\alpha$-good} if
\[
\deg_{\hk\setminus \hf}(J)\leq \alpha n^{k-|J|}.
\]
 If each vertex is $\alpha$-good, i.e., $\delta(\hf)>(1-\alpha)n^{k-1}$, then we say $\hf$ is $\alpha$-good with respect to $\hk$. Moreover, we call a  $j$-set $\alpha$-typical if every non-empty subset of it are $\alpha$-good. For a set pair  $(L,R)$, we say it is $\alpha$-good or typical if both $L,R$ are $\alpha$-good or typical.

In the following context, we set $\alpha'=\sqrt{2^{k}\alpha}$. The following proposition shows that almost all $j$-sets are $\alpha'$-typical.
\begin{prop}\label{prop:3.3}
	Given $1\leq j\leq k$ and  $0<\alpha<1$. Suppose that $\hf\subset \hk$ is $\alpha$-good with respect to $\hk$.  Then  the number of $j$-sets that are  not $\alpha'$-typical  is at most $\alpha' n^j$.
\end{prop}
\begin{proof}
	 Let $m_j$ be the number of $j$-sets that are not $\alpha'$-good. It is clear that $m_1=0$.
	Since $\hf$ is $\alpha$-good with respect to $\mathcal{K}$, there are at most $\alpha n^k$ edges in $\mathcal{K}\setminus \hf$. It follows that
	\[
	m_j\alpha'n^{k-j}\leq \binom{k}{j}|\mathcal{K}\setminus \hf|\leq \binom{k}{j}\alpha n^k.
	\]
	Then $m_j\leq \binom{k}{j}\alpha n^j/\alpha' $.  Let $N$ be the number of $j$-sets that are not $\alpha'$-typical. Then,
	\[
	N\leq \sum_{i=1}^j m_i n^{j-i}\leq \sum_{i=1}^j \binom{k}{i}\alpha n^j/\alpha' \leq 2^k\alpha n^j/\alpha'
	\]
	Thus $N\leq \alpha' n^j$.
\end{proof}

We need the following result, whose proof is postponed to Subsection \ref{proof:lem3.3}.

\begin{lem}
	\label{lem:k-graph}
	Given $1\leq \ell\leq k-1$. There exists $\alpha>0$ and $n_0$ such that the following holds. Suppose that $\hf\subseteq \mathcal{K}$ with $\delta(\hf)>(1-\alpha)n^{k-1}$ and $n\geq n_0$. Then  for any  $\alpha$-typical set pair $(L,R)$ with $L\cup R\in \hk$, there exists a  Hamilton $(\ell,k-\ell)$-path in $\hh$ with ends  $L$ and $R$.
\end{lem}

\begin{proof}[Proof of Lemma \ref{lem-stability} ]
	
Suppose that $\hg$ is a $k$-graph on $n$ vertices and $(L,R)$ be the set pair that is $\alpha$-good  with respect to $\hb$ described in Lemma \ref{lem-stability}. Let $A_1=A\setminus (L\cup R)$, $B_1=B\setminus (L\cup R)$,   and let $m=|A_1\cup B_1|/k$. Let $|A_1|:=k_1m+s$ with $0\leq s< m$.  Then $|A_1|\geq |A|-k\geq 0.45n-k$ and $k\geq 7$ imply $2\leq k_1\leq k-3$.

 Recall that $f(\hb)=0$,  then by \eqref{eq-3.2},  $f(\hb-L\cup R)=f(\hb)=0$, that is
\[
s+k_1m\equiv \eta(\hb-L\cup R) m=\eta(\hb)m \pmod 2.
\]
Then,
\begin{align}\label{eq:new3.4}
	s\equiv\begin{cases}
		 2m \mbox{ or } 2m+2 & \mbox{if} \quad 2|(k_1-\eta(\hb))\\
		3m \mbox{ or } 3m+2 & \mbox{if} \quad 2\nmid  (k_1-\eta(\hb))
	\end{cases}
	\pmod 4.
\end{align}
	
We are looking for a partition of $V(\hg)=(L\cup R)\dot{\cup}(X_1\cup\cdots\cup X_k)\dot{\cup}(Y_1\cup\cdots\cup Y_k)\dot{\cup} E$  satisfy some conditions. According to \eqref{eq:new3.4}, we distinguish four cases.
	
\begin{itemize}
	\item [(1)]  If $2|(k_1-\eta(\hb))$, $ s\equiv 2m \pmod 4$.
	
	Let $E$ be an empty set. Partition $A\setminus(L\cup R)$ into sets
	$X_1,\ldots,X_{k_1-2}, Y_1,\ldots,Y_{k_1+2}$ and  partition $B\setminus (L\cup R)$ into sets $X_{k_1-1},\ldots,X_k,Y_{k_1+3},\ldots,Y_k$ such that $|X_i|=\frac{2m-s}{4}$, $|Y_i|=\frac{2m+s}{4}$.
	
	\item [(2)] If $2|(k_1-\eta(\hb))$, $ s\equiv 2m+2 \pmod 4$.
	
	We first choose $E$ be an $(\ell,k-\ell)$-path with ends $(L_1,R_1)$ such that (i) $L_1,R_1\subset V(\hg)\setminus(L\cup R)$ are $\alpha'$-good and $L_1\cup R_1\in \hk$, (ii) $|E\cap A|=2 \pmod 4 $, (iii) $E=L_1R_1$ if $\eta(\hb)=0$;  $E=L_1R_2L_2R_1$ and $L_1R_2,L_2R_1\in\hg\cap\hb$ if $\eta(\hb)=1$. By Proposition \ref{prop:3.3}, almost all sets are $\alpha'$-good, then such $E$ exists. Then, partition $A\setminus(L\cup R\cup E)$ into sets
	$X_1,\ldots,X_{k_1-2},Y_1,\ldots,Y_{k_1+2}$ and partition $B\setminus (L\cup R\cup E)$ into sets $X_{k_1-1},\ldots,X_k;Y_{k_1+3},\ldots,Y_k$  such that $|X_i|=\frac{2m-(s-|E\cap A|)}{4}$, $|Y_i|=\frac{2m+(s-|E\cap A|)}{4}$.
	
	\item [(3)] If $2\nmid(k_1-\eta(\hb))$, $ s\equiv 3m \pmod 4$.
	
	Let $E$ be an empty set. Partition $A\setminus(L\cup R)$ into sets
	$X_1,\ldots,X_{k_1-1}, Y_1,\ldots,Y_{k_1+3}$ and  partition $B\setminus (L\cup R)$ into sets $X_{k_1},\ldots,X_k,Y_{k_1+4},\ldots,Y_k$ such that $|X_i|=\frac{3m-s}{4}$, $|Y_i|=\frac{m+s}{4}$.
	
	\item [(4)] If $2\nmid(k_1-\eta(\hb))$, $ s\equiv 3m+2 \pmod 4$.
	
	We first choose $E$ be an $(\ell,k-\ell)$-path with ends $(L_1,R_1)$ such that (i) $L_1,R_1\subset V(\hg)\setminus(L\cup R)$ are $\alpha'$-good and $L_1\cup R_1\in \hk$, (ii) $|E\cap A|=2 \pmod 4 $, (iii) $E=L_1R_1$  if $\eta(\hb)=0$;  $E=L_1R_2L_2R_1$ and $L_1R_2,L_2R_1\in\hg\cap\hb$ if $\eta(\hb)=1$. By Proposition \ref{prop:3.3}, almost all sets are $\alpha'$-good, then such $E$ exists. Then, 
partition $A\setminus(L\cup R\cup E)$ into sets
	$X_1,\ldots,X_{k_1-1},Y_1,\ldots,Y_{k_1+3}$ and partition $B\setminus (L\cup R\cup E)$ into sets $X_{k_1},\ldots,X_k;Y_{k_1+4},\ldots,Y_k$  such that $|X_i|=\frac{3m-(s-|E\cap A|)}{4}$, $|Y_i|=\frac{m+(s-|E\cap A|)}{4}$.
	
\end{itemize}	
Clearly, in each of the above four cases  $\hk(X_1,\ldots,X_k)$ and $\hk(Y_1,\ldots,Y_k)$ are both sub $k$-graph of $\hb$, and  $|\cup_{i=1}^kX_i|\geq \frac{n}{5}$, $|\cup_{i=1}^kY_i|\geq \frac{n}{5}$. By Proposition \ref{prop:3.2}, each vertex in $|\cup_{i=1}^kX_i|$ (or $|\cup_{i=1}^kY_i|$)  is $5\alpha$-good with respect to $\hk(X_1,\ldots,X_k)$ (or $\hk(Y_1,\ldots,Y_k)$, respectively).

Since $L,R$ are $\alpha$-good and $L_1,R_1$ are $\alpha'$-good if $E$ is not empty set, by Proposition \ref{prop:3.3},  there are $\alpha'$-typical sets  $L_1^*,R_1^*,L_2^*,R_2^*$ such that
(i) $L_1^*\cup R_1^*\in \hk(X_1,\ldots,X_k)$, $L_2^*\cup R_2^*\in \hk(Y_1,\ldots,Y_k)$;
(ii)  $LR_1^*, L_2^*R\in \hg\cap\hb$; (iii) $L_1^*R_1,L_1R_2^*\in \hg\cap\hb$ if $E\neq \emptyset$; $L_1^*R_2^*\in \hg\cap\hb$ if $E=\emptyset$.

 By Lemma \ref{lem:k-graph},   $\hk(X_1,\ldots,X_k)$  contains a Hamilton $(\ell,k-\ell)$-path $\hp_1$ with ends $(L_1^*,R_1^*)$;  $\hk(Y_1,\ldots,Y_k)$  contains Hamilton $(\ell,k-\ell)$-path $\hp_2$ with ends $(L_2^*,R_2^*)$. Therefore there exists Hamilton $(\ell,k-\ell)$-path $L\hp_1\hp_2R$ if $E=\emptyset$, $L\hp_1E\hp_2R$ if $E\neq \emptyset$.
\end{proof}

\subsection{Proof of Lemma \ref{lem:k-graph}}\label{proof:lem3.3}

Note that an $(\ell,k-\ell)$-path with ends $(L,R)$ is also an $(\ell,k-\ell)$-path with ends $(R,L)$. Thus we may assume that $1\leq \ell\leq k/2$ in the proof of Lemma \ref{lem:k-graph}.  We prove the lemma  by induction on $k$. The case $k=2$ was proved by Moon and Moser \cite{Moon1963}. We assume that the lemma holds for $k-1$ and prove it for $k$.

Let $(L,R)$ be the given $\alpha$-typical sets. By reordering the subscripts, we assume that $L\in\mathcal{K}(X_1,\ldots,X_\ell)$ and $R\in\mathcal{K}(X_{\ell+1},\ldots,X_k)$.

\begin{claim}\label{claim:3.2}

	 Given $q=n/16k$. There is a family $\ha=\mathcal{L}\cup \mathcal{R}^{-}\cup\mathcal{R}$ where $\mathcal{L}=\{L_1,L_2,\ldots,L_{2q} \}$, $\mathcal{R}=\{R_1,R_3,\ldots,R_{2q-1}\}$, $\mathcal{R}^-=\{R_2^-,R_4^-,\ldots,R_{2q}^-\}$  such that  $\hl\subset  \mathcal{K}(X_{1},\ldots,X_{\ell})$,  $\hr\subset  \mathcal{K}(X_{\ell+1},\ldots,X_{k}) $, $\hr^-\subset \mathcal{K}(X_{\ell+1},\ldots,X_{k-1}) $,  and for $i=1,\ldots,q$ the following hold.
	\begin{itemize}
		\item [(i)] for each $R_{2i}^-\in \mathcal{R}^-$ both  $L_{2i-1}R_{2i}^-$ and $L_{2i}R_{2i}^-$ are  $\alpha'$-good,
		\item [(ii)]  $L_{2i-1}R_{2i-1}, R_{2i-1}L_{2i}\in \hf$,
		\item [(iii)] for each $x\in X_{k}$, there are at least $q/2$ sets in $\mathcal{R}^-$ such that $L_{2i-1}R_{2i}^-x\in \hf$ and $xR_{2i}^-L_{2i}\in \hf$,
		\item [(iv)] $V(\ha)\cap (L\cup R)=\emptyset$.
	\end{itemize}
\end{claim}
\begin{proof}
	We use  $LR^-L'$ to denote a term that  $L,L'\in \mathcal{K}(X_{1},\ldots,X_{\ell})$,  $R^-\in \mathcal{K}(X_{\ell+1},\ldots,X_{k-1})$, and $L,L'$  are disjoint.  Let
	\[
	\hht:=\{LR^-L'\colon L,L',LR^-,R^-L' \mbox{ are } \alpha'\mbox{-good}  \}.
	\]
 By the setting of Lemma \ref{lem:k-graph},   all vertices are $\alpha$-good.  By Proposition \ref{prop:3.3},  the number of $j$-sets  that are not $\alpha'$-good is at most  $\alpha'n^j$.
	By the union bound, there are at most $4\alpha'n^{k+\ell-1}$ $LR^-L'$ such that at least one of   $L,L', LR^-, R^-L'$ are not $\alpha'$-good. Therefore, for $n$ sufficient large
	\begin{align}\label{ineq-3.5}
		|\hht|\geq n^{k-1}(n-1)^{\ell}-4\alpha'n^{k+\ell-1}\geq (1-5\alpha')n^{k+\ell-1}.
	\end{align}
 For each $x\in X_k$, the number of $(k-1)$-sets in $\hk(X_1,\ldots,X_{k-1})\setminus \mathcal{N}(x)$  is at most $\alpha n^{k-1}$. Then there are at most $\alpha'n^{k+\ell-1}$ $LR^-L'$ such that $LR^-x\notin \hf$ and at most $\alpha'n^{k+\ell-1}$ $LR^-L'$ such that $xR^-L'\notin \hf$. Let
 \[
 \hht(x):=\{ LR^-L'\colon  LR^-L'\in \hht, LR^-x,xR^-L'\in \hf\}.
 \]
Then by \eqref{ineq-3.5}
\begin{align}\label{eq:4.2}
	|\hht(x)|\geq (1-7\alpha')n^{k+\ell-1} \quad \mbox{for each } x\in X_k.
\end{align}
	
Let $\hm'$ be a random family  obtained by choosing each  element in $\{LR^-L'\colon L,L'\in \mathcal{K}(X_{1},\ldots,X_{\ell}),  R^-\in \mathcal{K}(X_{\ell+1},\ldots,X_{k-1}) \}$ independently with probability $p:=\frac{2q}{n^{k-1}(n-1)^\ell}$. Then $\ex(|\hm'|)=2q$ and
	\[
	\ex(|\hm'\cap \hht(x)|)\geq 2q(1-7\alpha') \quad \mbox{for every } x\in X_k.
	\]
	Since $n$ is sufficiently large and $\alpha'$ tends to $0$,  Proposition \ref{prop:1.2} implies that with high probability (close to $1$) for arbitrarily real $\gamma>0$ we have
	\begin{align}\label{eq:3.4}
		|\hm'\cap \hht(x)|\geq 2(1-\gamma)q \quad \mbox{for every} \quad x\in X_k
	\end{align}
and
	\begin{align}\label{eq:3.5}
	|\hm'|\leq (1+\gamma)\ex(|\hm'|)=2(1+\gamma)q.
\end{align}

	Let $Y$ be the number of intersecting pairs of members in $\hm'$. Let $Z$ be the number of   members that intersect  $L\cup R$. Then
	\begin{align}\label{eq:3.6}
	\ex(Y)\leq\frac{(2q)^2}{2}\frac{k+\ell-1}{n}\leq \frac{4kq^2}{n}= \frac{q}{4},
    \end{align}
and
\begin{align}\label{eq:3.7}
	\ex(Z)\leq \frac{2q ({k+a-1})}{n}\leq \frac{4kq}{n}=\frac{1}{4}.
\end{align}

	By Markov's bound, the probability that $Y\leq q/2$ is at least $1/2$, the probability that $Z< 1$ (i.e. $Z=0$) is at least $3/4$. Therefore we can find a family $\hm'$ such that \eqref{eq:3.4}--\eqref{eq:3.7} hold. By removing one set from each of the intersecting pairs in $\hm'$ and  removing the elements intersect $L\cup R$, we obtain a family $\hm$ such that
	\[
	 |\hm\cap \hht(x)|\geq 2(1-\gamma)q-Y-Z\geq (3/2-2\gamma)q.
	\]
	By \eqref{eq:3.4} and \eqref{eq:3.5}, we know that $|\hm'|-|\hm'\cap\hht(x)|\leq 2\gamma q$, therefore  $|\hm|-|\hm\cap\hht(x)|\leq 2\gamma q$.  By Removing some extra $LR^-L'$, we may assume $|\hm|=q$. Let $\hm=\{L_1R_2^-L_2,L_3R_4^-L_4,\ldots,L_{2q-1}R_{2q}^-L_{2q}\}$ and let $\mathcal{L}$ and $\mathcal{R}^-$ be the corresponding sets. Since $L_i$ are $\alpha'$-good, we can choose a family $\{R_1,R_3,\ldots,R_{2q-1}\}\subset \mathcal{K}(X_{\ell+1},\ldots,X_{k})$ of pairwise disjoint sets from the remaining vertex such that  $R_{2i-1}\cap (L\cup R)=\emptyset$, $i=1,2,\ldots,q$  and (ii) holds.
\end{proof}

Since $R_{1},R$ are both $\alpha'$-good, we can find an $L_{0}\in \mathcal{K}(X_{1},\ldots,X_{\ell})$ such that $R_1L_0$, $L_0R\in \hf$. Let $\ha^*=\ha\cup \{L_{0}, R\}$.  For $i\in [k]$, let $X_i'=X_i\setminus V(\ha^*)$, then $|X_i|=n-2q-1$.  Let $\hf^-\subseteq \mathcal{K}(X_1',X_2',\ldots,X_{k-1}') $ be the family  consisting of all $(1/128k)$-good $(k-1)$-sets. Let $\alpha_{k-1}>k2^{k+6}\sqrt{\alpha}$. Since $\delta(\hf)\geq (1-\alpha)n^{k-1}$,  we claim that
	 \[
	 \delta(\hf^-)\geq (1-\alpha_{k-1})(n-2q-1)^{k-2},
	 \]
	 i.e. each vertex in $V(\hf^-)$ is $ \alpha_{k-1}$-good with respect to $\mathcal{K}(X_1',X_2',\ldots,X_{k-1}')$. Indeed, if there exists $v$ such that $\deg_{\hf^-}(v)<(1-\alpha_{k-1})(n-2q-2)^{k-2}$, then the number of edges in $\hk\setminus \hf$ containing $v$ is at least
	  \begin{align*}
	 	\frac{\alpha_{k-1}(n-2q-1)^{k-2}n}{128k}\geq \alpha_{k-1}\frac{(1-\frac{1}{7k})^{k-2}}{128k}n^{k-1}>\alpha n^{k-1},
	 \end{align*}
a contradiction.

Since $L_{2q}$  is $\alpha'$-good and almost all $(k-\ell-1)$-sets are $\alpha'$-good, there are at least half of $(k-\ell-1)$-sets $S$ in $\mathcal{K}(X_{\ell+1},\ldots,X_{k-1})$ such that $L_{2q}\cup S$  are $(2\alpha')$-good. By Proposition \ref{prop:3.3}, almost all $j$-set are $\alpha'$-typical. Then we can find an $R_{2q+1}^-\in \mathcal{K}(X_{\ell+1},\ldots,X_{k-1})$ is $\alpha'$-typical such that $L_{2q}\cup R_{2q+1}^-$ is $2\alpha'$-good. We claim that $(L,R_{2q+1}^-)$ is  $\alpha_{k-1}$-typical with respect to $\mathcal{K}(X_1',X_2',\ldots,X_{k-1}')$. Otherwise, if there exists a $j$-set  $J\subset L$ or $R_{2q+1}^-$ not $\alpha_{k-1}$-good with respect to $\mathcal{K}(X_1',X_2',\ldots,X_{k-1}')$. Then the number of edges in $\hk\setminus \hf$  containing $J$ is at least
\begin{align*}
	\alpha_{k-1}(n-2q-1)^{k-1-|J|}\cdot n/128k>\alpha' n^{k-|J|},
\end{align*}
contradicting the fact that $J$ is $\alpha'$-good with respect to $\hf$.   Since $\alpha_{k-1}$ tends to $0$ as $\alpha$ tends to $0$, by the induction assumption, there exists a Hamilton $(\ell,k-\ell-1)$-path in $\hf^-$ with ends  $L$ and $R_{2q+1}^-$.  This path gives  a partition of $V(\hf^-)$:
\[ R_{2q+1}^-\cup L_{2q+1}\cup R_{2q+2}^-\cup L_{2q+2}\cup\cdots\cup R_{n-2}^-\cup  L_{n-2}\cup R_{n-1}^-\cup L.\]

Recall that $X_k'=X_k\setminus V(\ha^*)$, and let $\hr_1^-=\{R_{2q+1}^-,R_{2q+2}^-,\ldots,R_{n-1}^-\}$.    Let $B(X_k',\hr_1^-) $ be the bipartite graph with the edge set
\begin{align}\label{eq:new3.10}
	\{(x,R_i^-)\in X_k'\times \hr_1^-\colon   L_{i-1}R_{i}^-x, xR_i^-L_i\in \hf \},
\end{align}
where $i=2q+1,\ldots,n-1$ and $L_{n-1}=L$. Recall that every edge in $\hf^-$ is  $(1/128k)$-good with respect to $\hf$. We infer that  there are at least $|X_k'|-\frac{n}{64k}$
common neighborhoods in $X_k'$ between $L_{i-1}R_i^-$ and $R_i^-L_i$. That is, $\deg_{B(X_k',\hr_1^-)}(R_i^-)\geq |X_k'|-\frac{n}{64k}$. Therefore, there are  at most
$\frac{n}{64k}|\hr_1^-|/\frac{|\hr_1^-|}{2}=n/32k=q/2$ vertices in $X_k'$ has degree  less than $\frac{1}{2}|\hr_1^-|$.  We  choose a set $Y_{k}'=\{x_1,\ldots,x_{q}\}\subset X_k'$  containing all the vertices with degree  less than $\frac{1}{2}|\hr_1^-|$ and let $Z_k'=X_k'\setminus Y_k'$. Define another bipartite graph $B(Y_k',\hr^-)$ with the edge set
\[
\{(x,R_i^-)\in Y_k'\times \hr^-\colon   L_{2i-1}R_{2i}^-x, xR_{2i}^-L_{2i}\in \hf \}.
\]
By Claim \ref{claim:3.2} (iii) we infer $\deg_{B(Y_k',\hr^-)}(x)\geq q/2$ for all $x\in Y_k'$. By Claim \ref{claim:3.2} (i) $L_{2i-1}R_{2i}^-$ and $R_{2i}^-L_{2i}$ are both $\alpha'$-good. Then $\alpha'n<q/2$ implies that $\deg_{B(Y_k',\hr^-)}(R_{2i}^-)\geq q/2$ for all $R_{2i}^-\in \hr^-$. By Hall's Theorem, there is a perfect matching in $B(Y_k',\hr^-)$. Thus, there is a permutation of $x_1,x_2,\ldots,x_q$ such that
\begin{align}\label{eq:absorb-path}
	R_1L_1(R_2x_{i_1})L_2R_3L_3(R_4x_{i_2})L_4\cdots L_{2q-1}(R_{2q}x_{i_q})L_{2q}
\end{align}
is an $(\ell,k-\ell)$-path.

Since $|Z_k'|=|X_k'|-q=n-2q-1$, $|Y|=n-2q-1$,  $B(Z_k',\hr_1^-)$ is balanced.  Recall that  $\deg_{B(Z_k',\hr_1^-)}(x)\geq \frac{n-2q-1}{2}$ for all $x\in Z_k'$. By the definition of $\hf^-$, we know that $L_{i-1}R_{i}^-$, $R_i^-L_i$ are both $(1/128k)$-good. It follows that $\deg_{B(Z_k',\hr_1^-)}(R_i^-)\geq n-2q-1-\frac{n}{64k}\geq \frac{n-2q-1}{2}$ for all $R_i^-\in \hr_1^-$.
Then by Hall's Theorem there exists a perfect matching in $B(Z_k',\hr_1^-)$, which leads to an   $(\ell,k-\ell)$-path
\begin{align}\label{eq:new3.12}
	R_{2q+1} L_{2q+1} R_{2q+2} L_{2q+2}\cdots R_{n-2} L_{n-2} R_{n-1}  L,
\end{align}
where $R_i=R_i^-x$ for some $x\in Z_k'$ and $i=2q+1,\ldots,n-1$.

By \eqref{eq:new3.10}, we know that $L_{2q}R_{2q+1}\in \hf$. Moreover, $RL_0, L_0R_1\in \hf$, together with \eqref{eq:absorb-path} and \eqref{eq:new3.12}, we find a Hamilton $(\ell,k-\ell)$-path in $\hf$ with ends $L$ and $R$ as follows:
\[
RL_0R_1L_1(R_2x_1)L_2R_3L_3(R_4x_2)L_4\cdots L_{2q-1}(R_{2q}x_q)L_{2q}R_{2q+1} L_{2q+1} R_{2q+2} L_{2q+2}\cdots R_{n-2} L_{n-2} R_{n-1}  L.
\]

\end{document}